\documentclass[12pt]{amsart}

\usepackage{amsmath,amsthm,amssymb,graphicx,tikz,setspace,verbatim}
\usepackage[margin=1in]{geometry}
\usepackage[hidelinks]{hyperref}

\newtheorem{theorem}{Theorem}[section]
\newtheorem{lemma}[theorem]{Lemma}
\newtheorem{proposition}[theorem]{Proposition}
\newtheorem{corollary}[theorem]{Corollary}

\theoremstyle{definition} 
\newtheorem{definition}[theorem]{Definition}
\newtheorem{example}[theorem]{Example}
\newtheorem{remark}[theorem]{Remark}
\newtheorem{conjecture}[theorem]{Conjecture}

\DeclareMathOperator{\Cob}{2Cob}

\DeclareMathOperator{\ext}{ext}

\DeclareMathOperator{\gen}{gen}
\DeclareMathOperator{\id}{id}
\DeclareMathOperator{\im}{im}
\DeclareMathOperator{\inn}{in}
\DeclareMathOperator{\intt}{int}

\DeclareMathOperator{\open}{open}
\DeclareMathOperator{\out}{out}

\DeclareMathOperator{\Rep}{Rep}
\DeclareMathOperator{\rk}{rank}

\DeclareMathOperator{\spann}{span}

\DeclareMathOperator{\SAlg}{SAlg}

\newcommand{\C}{\mathbb{C}}
\newcommand{\Dc}{\mathcal{D}}
\newcommand{\F}{\mathbb{F}}
\newcommand{\gl}{\mathfrak{gl}}
\newcommand{\Lc}{\mathcal{L}}
\newcommand{\Pc}{\mathcal{P}}
\newcommand{\psl}{\mathfrak{psl}}
\newcommand{\Q}{\mathbb{Q}}

\newcommand{\Sc}{\mathcal{S}}

\newcommand{\Z}{\mathbb{Z}}
\newcommand{\Zb}{\mathbf{Z}}
\newcommand{\Zc}{\mathcal{Z}}

\title[Surface gluing with signs and gradings]{Surface gluing with signs and gradings in decategorified Heegaard Floer theory}
\author[Andrew Manion]{Andrew Manion}
\address{Department of Mathematics, North Carolina State University, 2108 SAS Hall, Raleigh, NC 27695}
\email{ajmanion@ncsu.edu}

\begin{document}

\begin{abstract}
    A previous result about the decategorified bordered (sutured) Heegaard Floer invariants of surfaces glued together along intervals, generalizing the decategorified content of Rouquier and the author's higher-tensor-product-based gluing theorem in cornered Heegaard Floer homology, was proved only over $\F_2$ and without gradings. In this paper we add signs and prove a graded version of the interval gluing theorem over $\Z$, enabling a more detailed comparison of these aspects of decategorified Heegaard Floer theory with modern work on non-semisimple 3d TQFTs in mathematics and physics.
\end{abstract}

\maketitle

\tableofcontents

\section{Introduction}

In \cite{ManionRouquier}, Rapha{\"e}l Rouquier and the author reformulated and extended work of Douglas--Manolescu \cite{DM} to prove a more structured gluing theorem for the algebras that bordered sutured Heegaard Floer homology \cite{LOTBorderedOrig,Zarev} assigns to surfaces. The gluing theorem of \cite{ManionRouquier} is based on a new tensor product construction for higher representations of a particular dg monoidal category  \cite{KhovOneHalf} related to $\mathfrak{gl}(1|1)^+$.  Given such a gluing theorem, it is natural to look for a TQFT interpretation. Some preliminary remarks to this end, focusing on the connection with 2d open-closed cobordisms, appear in \cite[Section 7.2.5]{ManionRouquier}. 

More can be said once one passes from the full version of bordered sutured Heegaard Floer theory to its decategorification, in which the higher tensor products become ordinary tensor products of representations of a Hopf algebra. Indeed, in \cite{ManionDHA}, the author proved a gluing theorem for decategorified bordered sutured Heegaard Floer invariants of surfaces, generalizing the decategorification of the gluing theorem from \cite{ManionRouquier} and placing it in a more flexible and TQFT-like setting. However, the results of \cite{ManionDHA} were only formulated and proved without gradings and with coefficients in $\F_2$, where one can ignore plus and minus signs (this was reasonable in relationship with \cite{ManionRouquier} where the grading structure does not readily lead to a decategorification over $\Z$). 

One reason to add gradings and signs to the results of \cite{ManionDHA} is to bring decategorified bordered sutured Heegaard Floer theory more in line with the recently active area of non-semisimple 3d TQFT \cite{ADO, CGPM, Mikhaylov, BCGPM, GPV, AGPS, GPPV, GHNPPS, CGP-NSS, Jagadale, GeerYoung}, especially the theory called $\psl(1|1)$ Chern--Simons theory by Mikhaylov \cite{Mikhaylov} and whose homological truncation is given a proposed mathematical description in Geer--Young's $\Dc^{q,\intt}$ TQFT for generic $q$ \cite{GeerYoung}. Based on the connection between Heegaard Floer homology and $\gl(1|1)$, as explained from a unified perspective in \cite[Section 3.3]{GPV}, one hopes that ideas from decategorified bordered sutured Heegaard Floer theory can advance the study of these non-semisimple 3d TQFTs and vice-versa. However, it is difficult to compare these TQFTs with \cite{ManionDHA} directly; they are defined over $\C$ rather than $\F_2$, and gradings are a key aspect of \cite{Mikhaylov,GeerYoung} but are absent in \cite{ManionDHA}. See \cite[Section 5]{CGP-NSS} for a discussion of the important role played by gradings in non-semisimple 3d TQFT in general.

The main result of this paper is a version with gradings and signs of the generalized gluing theorem from \cite{ManionDHA}. To state it, let $\Cob^{\ext}$ denote Lauda--Pfeiffer's open-closed cobordism category \cite{LaudaPfeiffer}; informally, objects $M$ of $\Cob^{\ext}$ are disjoint unions of oriented circles and intervals, and morphisms are 2d open-closed cobordisms. If $F \colon M_1 \to M_2$ is such a morphism, we will write\footnote{As mentioned in \cite{ManionRouquier, ManionDHA} and reviewed below, a morphism $F \colon M_1 \to M_2$ in $\Cob^{\ext}$ can be canonically identified with a sutured surface $(F,\Lambda,S_+,S_-)$ in the sense of Zarev \cite[Definition 1.2]{Zarev}, allowing closed components of $F$ as well as components of $\partial F$ that are disjoint from $\Lambda$ and contained fully in $S_+$ or $S_-$, and equipped with a partition of $S_+$ into ``source'' $M_1$ and ``target'' $M_2$ along with an ordering of the components of $M_i$ for $i \in \{1,2\}$.} $S_+ = M_2 \sqcup (-M_1)$ for the subset of $\partial F$ corresponding to the source $M_1$ and target $M_2$ of $F$. For an object $M$ of $\Cob^{\ext}$ we will write $A(M)$ for the tensor product, over all interval components of $M$, of $\Z$-graded\footnote{In general, when considering $\Z$- or $\mathbb{Q}$-graded super abelian groups in this paper, we will not let the grading affect the signs in symmetrizers; only the parity as a super abelian group will contribute.} super rings $\Z[E]/(E^2) = U^{\Z}(\gl(1|1)^+)$ (one for each interval component) where $E$ is odd and has $\Z$-degree $-1$.

Let $F \colon M_1 \to M_2$ be a morphism in $\Cob^{\ext}$ (we think of $M_1$ as being on the right and $M_2$ as being on the left). The decategorification over $\F_2$ of Zarev's bordered sutured Heegaard Floer invariant of $F$ is isomorphic to $\wedge^* H_1(F, S_+ ;\F_2)$. If we let $A^{\F_2}(M_i) := A(M_i) \otimes_{\Z} \F_2$, then decategorifying the higher actions of \cite{ManionRouquier} gives $\wedge^* H_1(F, S_+ ;\F_2)$ the structure of a bimodule over $(A^{\F_2}(M_2),A^{\F_2}(M_1))$ as shown in \cite{ManionDHA}, and we get a functor
\[
\wedge^* H_1(F,S_+;\F_2) \otimes_{A^{\F_2}(M_1)} -
\]
from the category of left modules over $A^{\F_2}(M_1)$ to the category of left modules over $A^{\F_2}(M_2)$. For a composition of morphisms in $\Cob^{\ext}$ that only involves gluing along intervals and not circles, by \cite[Theorem 1.3]{ManionDHA} the bimodule or functor associated to the glued surface is a tensor product or composition of the bimodules or functors associated to the two pieces.

Over $\Z$ we will consider the $\Q$-graded\footnote{When choosing the function $\delta \colon \Sc \to \Q$ below, one could decide to incorporate arbitrarily complicated fractions, or even real or complex numbers after changing the codomain of $\delta$, into its definition. However, the most motivated choices appear to involve only simple fractions like half-integers and $(1/4)$-integers when they involve fractions at all, and if one sticks to such choices then elements of $\Q$ with arbitrarily complicated denominators will not appear.} super abelian groups
\[
\Zb^{S_+}_{\delta,\pi}(F) := \left( \Z^{0|1} \right)^{\otimes \pi(F)} \otimes \wedge^* H_1(F, S_+;\Z) \{ \delta(F) \}
\]
(homology coefficients will be $\Z$ unless otherwise stated) where:
\begin{itemize}
    \item $\Z^{0|1}$ denotes $\Z$ as a purely odd super abelian group concentrated in $\Q$-degree zero;
    \item the exterior algebra without shifts lives in $\Q$-degrees $0,1,\ldots,h$ where 
    \[
    h := \rk H_1(F,S_+)
    \]
    (see equation \eqref{eq:hFormula} below for a combinatorial formula for $h$);
    \item $\delta\colon \Sc \to \Q$ and $\pi \colon \Sc \to \Z/2\Z$, where $\Sc$ is the set of sutured surfaces, are ``degree-shift'' and ``parity'' functions on $\Sc$, chosen from a parametrized family $\Pi$ of such functions that we will describe below;
    \item the brackets $\{ \cdot \}$ denote an upward shift of the $\Q$-degree by the enclosed quantity.
\end{itemize} 
We will define the structure of a bimodule over $(A(M_2),A(M_1))$ on $\Zb^{S_+}_{\delta,\pi}(F)$.

\begin{remark}
    We want to think of the parity shift as being mediated by an explicit prefactor $\left( \Z^{0|1} \right)^{\otimes \pi(F)}$ that appears on one side of $\wedge^* H_1(F,S_+)$ or the other, rather than just an abstract shift like $[\pi(F)]$, because to make the signs work out in the gluing theorem, odd algebra elements trying to act on $\wedge^* H_1(F,S_+)$ should pick up a sign or not based on whether they need to commute past the prefactor.
\end{remark}

\begin{theorem}\label{thm:MainIntro}
Let $M_1$, $M_2$, and $M_3$ be objects of $\Cob^{\ext}$, and assume that $M_2$ has only intervals and no circles. Let
\[
M_3 \xleftarrow{F'} M_2 \xleftarrow{F} M_1
\]
be morphisms in $\Cob^{\ext}$. For any choice in the family $\Pi$ of a pair $(\delta,\pi)$, we have
\[
\Zb^{S_+}_{\delta,\pi}(F' \circ F) \cong \Zb^{S_+}_{\delta,\pi}(F') \otimes_{A(M_2)} \Zb^{S_+}_{\delta,\pi}(F)
\]
as bimodules over $(A(M_3),A(M_1))$.
\end{theorem}

We will use Theorem~\ref{thm:MainIntro} to deduce the following corollary.

\begin{corollary}\label{cor:OpenTQFT}
Let $\Cob^{\ext}_{\open}$ denote the full symmetric monoidal subcategory of $\Cob^{\ext}$ on objects consisting only of intervals and no circles. For any $(\delta,\pi) \in \Pi$, the assignments $M \mapsto A(M)$ and $F \mapsto \Zb^{S_+}_{\delta,\pi}(F)$ give a symmetric monoidal functor $\Zb^{S_+}_{\delta,\pi}$ from $\Cob^{\ext}_{\open}$ to the symmetric monoidal category $\SAlg_{\Z}$ of $\Z$-graded super rings and $\Q$-graded bimodules up to isomorphism.
\end{corollary}

We briefly describe the family $\Pi$ of degree and parity shift functions, postponing a more detailed explanation until Section~\ref{sec:DegreeParityConstraints}. For a sutured surface $F$ with $h := \rk H_1(F,S_+)$, and for any $(A_1, A_2, A_3, A_4) \in \Q^4$, define
\begin{equation}\label{eq:DeltaDef}
\begin{aligned}
\delta_{A_1,A_2,A_3,A_4}(F) := &-A_1 h + (A_1 - 1)\#\{\textrm{no-}S_+\textrm{ non-closed components of }F\} \\ 
&+ A_2\#\{\textrm{no-}S_-\textrm{ non-closed components of }F\} \\ 
&+ A_3\#\{\textrm{closed components of }F\} + ((A_1 - 1)/2) \#\{S_+\textrm{ intervals}\} \\
&+ A_4\#\{S_+\textrm{ circles}\} .
\end{aligned}
\end{equation}
For any $(N_1,N_2,N_3,N_4) \in \{0,1\}^4$, define
\begin{equation}\label{eq:PiDef}
\begin{aligned}
\pi_{N_1,N_2,N_3,N_4}(F) &= h  + N_1\#\{\textrm{no-}S_-\textrm{ non-closed components of }F\} \\
&+ N_2\#\{\textrm{closed components of }F\} + N_3\#\{S_+\textrm{ intervals}\} \\
&+ N_4\#\{S_+\textrm{ circles}\} 
\end{aligned}
\end{equation}
modulo 2. Note that both $\delta_{A_1,A_2,A_3,A_4}$ and $\pi_{N_1,N_2,N_3,N_4}$ are additive under disjoint unions of sutured surfaces. An arbitrary element $(\delta_{A_1,A_2,A_3,A_4},\pi_{N_1,N_2,N_3,N_4})$ of $\Pi$ is specified by the choice of $(A_1,A_2,A_3,A_4)$ and $(N_1,N_2,N_3,N_4)$, which may be unrelated to each other.

\begin{figure}
    \centering
    \includegraphics[scale=0.9]{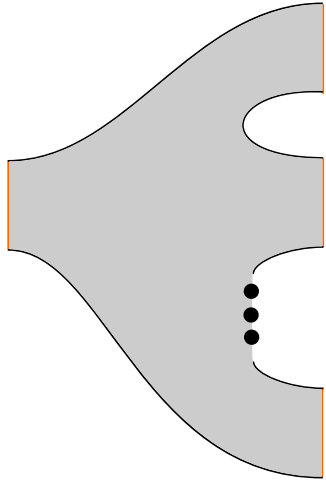}
    \caption{Open $p$-tuple of pants $\Pc$, with $p$ intervals on its right side and one on its left side.}
    \label{fig:OpenPants}
\end{figure}

\begin{example}\label{ex:TensorShifts}
    To the open $p$-tuple of pants $\Pc$ with $p \geq 0$ intervals in its input boundary (see Figure~\ref{fig:OpenPants}), we can ask which choices $(\delta_{A_1,A_2,A_3,A_4},\pi_{N_1,N_2,N_3,N_4}) \in \Pi$ result in $\Pc$ being assigned $(\Z[E]/(E^2))^{\otimes p}$ as a right module over itself, with left action of $\Z[E]/(E^2)$ given by iterating the coproduct 
    \[
    \Delta(E) = E \otimes 1 + 1 \otimes E
    \]
    (or the counit $\varepsilon(E) = 0$ when $p=0$). Note that if this identification holds, then taking
    \[
    \Zb^{S_+}_{\delta_{A_1,A_2,A_3,A_4},\pi_{N_1,N_2,N_3,N_4}}(\Pc) \otimes_{(\Z[E]/(E^2))^{\otimes p}} (X_1 \otimes_{\Z} \cdots \otimes_{\Z} X_p)
    \]
    for left modules $X_i$ over $A(M_1)$ yields the tensor product $X_1 \otimes_{\Z} \cdots \otimes_{\Z} X_p$ in the monoidal category of representations of the Hopf superalgebra $\Z[E]/(E^2)$, in line with the higher-tensor-product interpretation of the gluing theorem in \cite{ManionRouquier}.

    For $\Pc$ we have $h = p$; the one component of $\Pc$ is non-closed and intersects both $S_+$ and $S_-$ nontrivially. It has $p+1$ intervals in $S_+$ and no circles in $S_+$. We get
    \[
    \delta_{A_1,A_2,A_3,A_4}(\Pc) = -pA_1 + ((A_1 - 1)/2)(p+1) = -p(A_1 + 1)/2 + (A_1 - 1)/2
    \]
    and
    \[
    \pi_{N_1,N_2,N_3,N_4}(\Pc) = p + N_3(p+1).
    \]
    It will turn out (see Proposition~\ref{prop:OpenPantsTensor}) that we want the basis element of $\wedge^* H_1(\Pc,S_+)$ in the highest degree to correspond to $\pm 1 \in (\Z[E]/(E^2))^{\otimes p}$. Thus, we want $\delta_{A_1,A_2,A_3,A_4}(\Pc) = -p$ and (modulo 2) $\pi_{N_1,N_2,N_3,N_4}(\Pc) = p$. For any choice of parameters with
    \[
    A_1 = 1 \quad \textrm{and} \quad N_3 = 0,
    \]
    we get a grading shift that recovers tensor products in 
    $\Rep(\Z[E]/(E^2))$ from the open pair of pants $\Pc$ for all $p$. The choice $(A_1,A_2,A_3,A_4) = (1,0,0,0)$ and $(N_1,N_2,N_3,N_4) = (0,0,0,0)$ stands out as particularly simple; for general $F$ we have
    \[
    \delta_{1,0,0,0}(F) := - h \quad \textrm{and} \quad \pi_{0,0,0,0}(F) = h.
    \]
    If we are only concerned with the case $p=1$ (the identity cobordism on an interval) and not general $p$, then no conditions are required on $(A_1,A_2,A_3,A_4)$ and $(N_1,N_2,N_3,N_4)$ for this identity cobordism to get assigned $\Z[E]/(E^2)$ as a bimodule over itself.
\end{example}

\begin{remark}
As with the $\F_2$ case in \cite{ManionDHA}, the super abelian groups $\Zb^{S_+}_{\delta,\pi}(F)$ do not seem to admit an analogue of Theorem~\ref{thm:MainIntro} when gluing along circles rather than intervals. Relatedly, while $\Zb^{S_+}_{\delta,\pi}(F)$ has actions of $U^{\Z}(\gl(1|1)^+) = \Z[E]/(E^2)$ for interval components of $S_+$, it does not seem to have interesting actions of anything related to the full $\mathfrak{gl}(1|1)$ (with both odd generators $E$ and $F$ of $\gl(1|1)$ acting) given a circle component of $S_+$. In \cite{ManionOpenClosed}, we plan to address both limitations by introducing larger variants $\Zb^{P}_{\delta,\pi}(F)$ of the spaces $\Zb^{S_+}_{\delta,\pi}(F)$ such that $\Zb^{P}_{\delta,\pi}(F)$ admits interesting actions of $U(\mathfrak{psl}(1|1))$ for circle components of $S_+$ as well as $U(\psl(1|1)^+)$ for interval components of $S_+$. The spaces $\Zb^{P}_{\delta,\pi}(F)$ will satisfy analogues of Theorem~\ref{thm:MainIntro} when gluing along circles as well as intervals.
 \end{remark}

\bigskip

\noindent \textbf{Comparing with 3d non-semisimple TQFT.} Consider a connected sutured surface $F = F_{g,p}$ of genus $g$ with $p \geq 1$ boundary circles in $S_+$ and no other boundary (so $S_- = \varnothing$). We have $h = 2g - 1 + p$. Since $\partial F_{g,p} = S_+$, we have
    \[
    H_1(F,S_+) = H_1(F,\partial F) \cong H^1(F);
    \]
    in this section on the comparison with non-semisimple TQFT, but not elsewhere, we will work over $\C$ and by $H^1(F)$ we mean $H^1(F;\C)$.

    In \cite[Section 5.2.2]{Mikhaylov}, Mikhaylov looks at this sutured surface $F_{g,p}$ with the boundary circles viewed as punctures and labeled by line operators $L_{\mathbf{t_i},n_i}$. He assumes that the parameters $\mathbf{t_i} \in \C^{\times}$ are ``generic,'' i.e. not equal to $1$ (but that their product is equal to $1$), and that the parameters $n_i$ are integers. The parameters $\mathbf{t_i}$ (along with some choice of $2g$ additional holonomies in $\C^{\times}$ when $g > 0$) specify a nontrivial complex line bundle $\Lc$ with flat connection on $F_{g,p}$. Mikhaylov sets
    \[
    h_{\gen} = 2g - 2 + p = \dim H^1(F_{g,p};\Lc)
    \]
    (Mikhaylov calls this quantity $h$ but to avoid confusion we will reserve $h$ for the dimension of $H_1(F,S_+)$, which is $2g-1+p = h_{\gen}+1$ in this case). The reason the dimension $h_{\gen}$ for generic holonomies $\mathbf{t_i}$ is one less than our dimension $h$ is the relation shown graphically in \cite[Figure 2]{Mikhaylov}. Mikhaylov describes the 3d $\mathfrak{psl}(1|1)$ Chern--Simons TQFT as having state space
    \[
    (\det H^1(F_{g,p};\Lc))^{\otimes (-1/2 + N/h_{\gen})} \otimes \wedge^* H^1(F_{g,p};\Lc)
    \]
    on $F_{g,p}$ decorated by $\Lc$, where $N = \sum_{i=1}^p (-1/2 + n_i)$. See \cite[equation (5.6)]{Mikhaylov}. He gives the graded superdimension of this space as
    \begin{equation}\label{eq:MikhaylovGeneric}
    t^N (t^{1/2} - t^{-1/2})^{h_{\gen}}
    \end{equation}
    (see \cite[equation (5.8)]{Mikhaylov}) where powers of $t$ correspond to our $\Q$-grading on the state space, i.e. Mikhaylov's $U(1)_{\mathrm{fl}}$ charge. 
   
   Geer--Young note in \cite[Section 6.1]{GeerYoung} that, in their $\Dc^{q,\intt}$ TQFT, their graded dimension formula \cite[equation (19)]{GeerYoung} for the Hilbert spaces of generic surfaces agrees with Mikhaylov's formula \eqref{eq:MikhaylovGeneric}, at least up to the sign in front. More explicitly, when labeling the $p$ boundary components of $F_{g,p}$ by the representations $V(\alpha_i,n_i)_{\overline{0}}$ defined in \cite[Section 2.3.2]{GeerYoung}, and setting $N = \sum_i (n_i - 1/2)$ as in Mikhaylov's paper, Geer--Young's equation (19) gives the graded dimension of the state space as
    \begin{equation}\label{eq:GeerYoungGenericEvenLabels}
    (-1)^{g-1} t^{N} (t^{1/2} - t^{-1/2})^{h_{\gen}}
    \end{equation}
    where we write Geer--Young's $q^{\beta}$ as $t^{-1/2}$ (see \cite[Theorem 3.5]{GeerYoung}), and we take $t^{\mu} = 1$ because the product of the $\mathbf{t_i}$ is $1$.

    It is natural to ask how the graded super dimensions of our spaces compare with the above formulas, for various choices of degree shift and parity functions $(\delta,\pi) \in \Pi$. Our spaces are based on $H^1(F;\C)$ which arises when taking $\Lc$ to be the trivial local system, so we will take the limit $\mathbf{t}_i \to 1$ in Mikhaylov's analysis (and correspondingly $\alpha_i \to 0$ in Geer--Young) without worrying about possible subtleties involved with these limits as mentioned in \cite[Section 5.3.5]{Mikhaylov}. In other words, we will be concerned with the state spaces of non-generic or critical surfaces rather than generic surfaces. 
    
    If all of Mikhaylov's $\mathbf{t_i}$ are $1$, the relation in \cite[Figure 2]{Mikhaylov} gets trivialized. Where the graded dimension formulas in the generic case have $h_{\gen}$, we should have $h$ instead. We will also take all the labels $n_i$ to be zero, so that the punctures are each labeled by $V(0,0)_{\overline{0}}$; naively extrapolating \eqref{eq:GeerYoungGenericEvenLabels} leads to
    \begin{equation}\label{eq:GYGradedDimension1}
    (-1)^{g-1} t^{-p/2} (t^{1/2} - t^{-1/2})^{h}.
    \end{equation}
    Observations about circle gluing from a follow-up paper \cite{ManionOpenClosed} in preparation suggest that because we are in the non-generic case we should actually multiply \eqref{eq:GYGradedDimension1} by $-t^{1/2}$ and try to recover the graded dimension formula
    \begin{equation}\label{eq:ModifiedGradedDimension}
    (-1)^g t^{-p/2 + 1/2} (t^{1/2} - t^{-1/2})^{h}.
    \end{equation}
    For now, note that in \eqref{eq:ModifiedGradedDimension} as well as the generic formulas \eqref{eq:MikhaylovGeneric} and \eqref{eq:GeerYoungGenericEvenLabels}, all powers of $t$ are integral, but this is not true in \eqref{eq:GYGradedDimension1}. We want to characterize which choices of our parameters recover the graded dimension formula \eqref{eq:ModifiedGradedDimension}.

    \begin{itemize}
        \item To recover \eqref{eq:ModifiedGradedDimension}, we want to choose parameters so that our degree shift $\delta_{A_1,A_2,A_3,A_4}(F_{g,p})$ equals
        \[
        -p/2 + 1/2 - (1/2)(2g - 1 + p) = -h/2 - p/2 + 1/2.
        \]
        Thus, we should set
        \[
        A_1 = 1/2, \quad A_2 = 1/2, \quad A_4 = -1/2
        \]
        with $A_3$ left free to vary. Note that the condition $A_1 = 1/2$ is inconsistent with the condition $A_1 = 1$ from Example~\ref{ex:TensorShifts}. This is one reason we chose to look at a parametrized family of grading shifts in this paper rather than singling out one shift as the canonical choice, since we view both tensor products from open $p$-tuples of pants and connections to 3d non-semisimple TQFT as important in this theory.
        
        \medskip

        \item To recover \eqref{eq:ModifiedGradedDimension}, we also want our parity $\pi(F_{g,p})$ to be
        \[
        g  + (2g - 1 + p) = g + p = h/2 + p/2 - 1/2.
        \]
        Since the formulas for our parity functions $\pi_{N_1,N_2,N_3,N_4}$ all have a coefficient of $1$ rather than $1/2$ in front of $h$, there is no choice of $N_1$ that works. However, from the perspective of \cite{ManionOpenClosed} it will be natural to take
        \[
        (A_1, A_2, A_3, A_4) = (1/2, \, 1/2, \, 0, \, -1/2)
        \]
        in \eqref{eq:DeltaDef}, defining (for general open-closed cobordisms $F$) a shift
        \begin{align*}
        \delta_{1/2}(F) =& -h/2 - (1/2)\#\{\textrm{no-}S_+\textrm{ non-closed components of }F\} \\ 
        &+ (1/2)\#\{\textrm{no-}S_-\textrm{ non-closed components of }F\} \\
        & - (1/4) \#\{S_+\textrm{ intervals}\} - (1/2)\#\{S_+\textrm{ circles}\}.
        \end{align*}
        When the number of $S_+$ intervals of $F$ is equal modulo 4 to twice the number of boundary components of $F$ intersecting $S_-$ nontrivially, one can show that $\delta_{1/2}(F)$ is an integer, so we can let
        \[
        \pi_{1/2}(F) = \delta_{1/2}(F) \quad \textrm{modulo } 2.
        \]
        For $F = F_{g,p}$ we have no $S_+$ intervals and no boundary components of $F$ intersecting $S_-$ nontrivially, so $\pi_{1/2}(F_{g,p})$ makes sense and we have
        \[
        \pi_{1/2}(F_{g,p}) = -h/2 -p/2 + 1/2.
        \]
        This quantity equals $h/2 + p/2 - 1/2$ modulo $2$, so we see that the degree shift $\delta_{1/2}$ and parity $\pi_{1/2}$ combine to recover the graded dimension formula \eqref{eq:ModifiedGradedDimension}.
    \end{itemize}

\begin{conjecture}\label{conj:GeerYoungV}
When applying the Geer--Young TQFT $\Zb^{GY}$ to the surface $F_{g,p}$ for $p \geq 1$, assume we are viewing all boundary components of $F_{g,p}$ as punctures labeled by $V(0,0)_{\overline{0}}$. View the state space $\Zb^{GY}(F_{g,p})$ as $\Z$-graded by the second $\Z$ factor of Geer--Young's grading by $Z = \Z \times \Z \times \Z/2\Z$. For the degree-shift and parity functions $\delta_{1/2}$ and $\pi_{1/2}$ defined above, we have
\[
\Zb^{S_+}_{\delta_{1/2},\pi_{1/2}}(F_{g,p}) \cong \Zb^{GY}(F_{g,p})
\]
as $\Z$-graded super vector spaces.
\end{conjecture}
Assume that, like $\Zb^{S_+}_{\delta_{1/2},\pi_{1/2}}(F_{g,p})$, the $\Z$-graded super vector space $\Zb^{GY}(F_{g,p})$ is purely even or purely odd in each $\Z$-degree. Conjecture~\ref{conj:GeerYoungV} then amounts to saying that the graded dimensions of the two sides agree, and since by the choice of $\delta_{1/2}$ and $\pi_{1/2}$ the graded dimension of $\Zb^{S_+}_{\delta_{1/2},\pi_{1/2}}(F_{g,p})$ is given by \eqref{eq:ModifiedGradedDimension}, the conjecture says that \eqref{eq:ModifiedGradedDimension} is the correct extrapolation of \eqref{eq:GeerYoungGenericEvenLabels} from generic decorations to our non-generic decorations.

\begin{remark}
The connection between decategorified bordered Floer homology and exterior powers of $H_1$ groups first appeared in work of Petkova \cite{PetkovaDecat}. When $F$ is connected with one boundary component and $|\Lambda| = 2$, so that one of $\{C_1,C_2\}$ is an interval and the other is empty, Petkova defined a decategorification over $\Z$ of the bordered Heegaard Floer invariant of $F$ and identified it with $\wedge^* H_1(F;\Z)$, which given the assumptions is isomorphic to $\wedge^* H_1(F,S_+;\Z)$. See also Hom--Lidman--Watson \cite{HLW} for further connections to bordered Heegaard Floer theory at the 3-dimensional level.

Putting aside the connection to Heegaard Floer theory, vector spaces $\wedge^* H_1(F)$ also appear as state spaces of surfaces in earlier TQFT constructions of Frohman--Nicas \cite{FrohmanNicas} and Donaldson \cite{DonaldsonTQFT} related to the Alexander polynomial (see also Kerler \cite{KerlerHomologyTQFT}), and close relatives of these spaces appear in related contexts in many other places in the literature. As far as the author is aware, \cite{ManionDHA} and this paper are the first investigations of TQFT structure in dimensions 1 and 2 specifically for these wedge-of-first-homology state spaces (most TQFT references involving $\wedge^* H_1$ stick to closed surfaces and look at structure in dimensions 2 and 3).
\end{remark}

\medskip
\noindent \textbf{Organization.} In Section~\ref{sec:Surfaces} we discuss sutured surfaces and their relationship to open-closed cobordisms, and in Section~\ref{sec:Actions} we define actions of $\Z[E]/(E^2)$ on $\Zb^{S_+}_{\delta,\pi}(F)$ for sutured surfaces $F$. In Section~\ref{sec:Tensor} we look at the open $p$-tuple of pants example in more detail and examine its relationship with Hopf-algebraic tensor products. In Section~\ref{sec:MainLemma} we prove our main gluing lemma; in Section~\ref{sec:MainProof} we use this lemma to deduce Theorem~\ref{thm:MainIntro} and in Section~\ref{sec:CorollaryProof} we prove Corollary~\ref{cor:OpenTQFT}. In Section~\ref{sec:DegreeParityConstraints} we discuss the rationale behind the parametrized families \eqref{eq:DeltaDef} and \eqref{eq:PiDef}.

\medskip
\noindent \textbf{Acknowledgments.} The author would like to thank Sergei Gukov and Matthew Young for helpful conversations. The author was partially supported by NSF grant number DMS-2151786.

\section{Preliminaries}\label{sec:Preliminaries}

\subsection{Sutured surfaces and open-closed cobordisms}\label{sec:Surfaces}

We start by recalling Zarev's definition of sutured surface \cite[Definition 1.2]{Zarev}, generalized to remove some of the topological restrictions that Zarev imposes, and with some extra data that will be useful for signs in this paper.

\begin{definition}
A sutured surface is a quintuple $(F,\Lambda,S_+,S_-,\ell)$ where:
\begin{itemize}
\item $F$ is a compact oriented surface, possibly with boundary;
\item $\Lambda$ is a choice of some even number $\geq 0$ of points in each boundary component of $F$;
\item The components of $\partial F \setminus \Lambda$ are labeled as being in $S_+$ or $S_-$, in alternating fashion across the points of $\Lambda$.
\item Each component of $S_+$ is labeled as ``incoming'' or ``outgoing,'' and the sets of incoming and outgoing components of $S_+$ are each given an ordering (we let $\ell$ denote this labeling and ordering data).
\end{itemize}
\end{definition}
Unlike Zarev, we allow closed surfaces as well as circle boundary components of $F$ that are contained entirely in $S_+$ or $S_-$, and we require the choice $\ell$ of incoming/outgoing data and ordering data.

\begin{figure}
    \centering
    \includegraphics{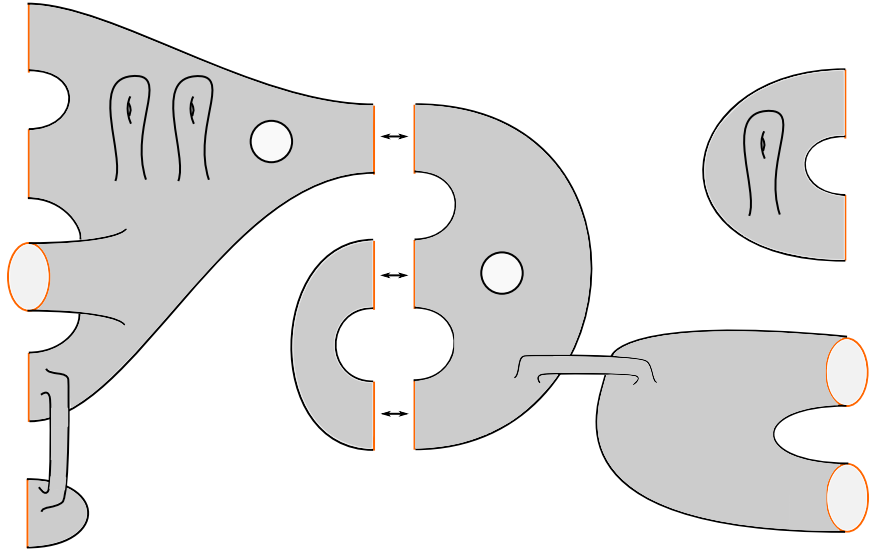}
    \caption{Two sutured surfaces (right and left) being viewed as morphisms in $\Cob^{\ext}$; to compose the morphisms one would glue as indicated by the arrows. In the sutured surface language, $S_+$ is drawn in orange and $S_-$ is drawn in black, while $\Lambda$ is the set of interface points between the orange and black boundary intervals. Viewing the surfaces as morphisms in $\Cob^{\ext}$, the source 1-manifold is drawn on the right (ordered from top to bottom) and the target 1-manifold is drawn on the left for each of the two surfaces. To simplify the figure we do not draw orientations.}
    \label{fig:CompositionExample}
\end{figure}

\begin{remark}
With the above definition, a sutured surface is the same data as a morphism in the open-closed oriented cobordism category $\Cob^{\ext}$ defined in \cite{LaudaPfeiffer}. We view $F$ as a morphism from $M_1$ to $M_2$ where $M_1$ is given by the orientation reversal of the incoming part of $S_+$ (a disjoint union, in order, of oriented intervals and circles) and $M_2$ is given by the outgoing part of $S_+$. We take orientations on all parts of the boundary of $F$ to be induced from the orientation on $F$ in the usual way without regard to sutured structure. With the boundary orientation we have $S_+ = M_2 \sqcup (-M_1)$ where the negative sign denotes orientation reversal.
\end{remark}

\begin{remark}
The reason we need ordering data is to specify unambiguously, for $F \colon M_1 \to M_2$ a morphism in $\Cob^{\ext}$, the left action of 
\[
A(M_1) := \Z[E]/(E^2) \otimes \cdots \otimes \Z[E]/(E^2)
\]
and the right action of $A(M_2)$ on the super abelian group we will associate to $F$. We need to know which factor of $\Z[E]/(E^2)$ corresponds to which interval.
\end{remark}

Figure~\ref{fig:CompositionExample} shows two sutured surfaces being presented as composable morphisms in the category $\Cob^{\ext}$. Our main gluing theorem, Theorem~\ref{thm:MainIntro}, will apply to the type of composition shown in Figure~\ref{fig:CompositionExample}, where the 1-manifold along which we are gluing consists only of intervals and no circles.

\subsection{Actions on exterior powers}\label{sec:Actions}

Let $(F,\Lambda,S_+,S_-,\ell)$ be a sutured surface. Choose 
\[
(A_1,A_2,A_3,A_4) \in \Q^4
\]
and let $\delta := \delta_{A_1,A_2,A_3,A_4}$ as defined in \eqref{eq:DeltaDef}. Similarly, choose 
\[
(N_1,N_2,N_3,N_4) \in \{0,1\}^4
\]
and let $\pi := \pi_{N_1,N_2,N_3,N_4}$ as defined in \eqref{eq:PiDef}. Consider the $\Z$-graded super abelian group
\[
\Zb^{S_+}_{\delta,\pi}(F) := \left( \Z^{0|1} \right)^{\otimes \pi(F)} \otimes \wedge^* H_1(F,S_+) \{\delta(F)\}
\]
where $\Z^{0|1}$ denotes $\Z$ viewed as a purely odd super abelian group. We write $\varepsilon_F$ for the standard basis element of $\left( \Z^{0|1} \right)^{\otimes \pi(F)}$; the parity of $\varepsilon_F$ is the same as the parity of $\pi(F)$. 

On $\Zb^{S_+}_{\delta,\pi}(F)$ we will define an action of the superalgebra $\Z[E]/(E^2)$ (with $E$ odd and of degree $-1$) for each interval component $I$ of $S_+$ by adding signs to the definition over $\F_2$ in \cite{ManionDHA}. It suffices to define, given $I$, a ($\Z$-linear) odd endomorphism $E$ of $\Zb^{S_+}_{\delta,\pi}(F)$ that has degree $-1$ and squares to zero. 

\begin{definition}
Let $(F,\Lambda,S_+,S_-,\ell)$ be a sutured surface and let $I$ be an outgoing interval component of $S_+$. We define
\[
E = E_I \colon \Zb^{S_+}_{\delta,\pi}(F) \to \Zb^{S_+}_{\delta,\pi}(F)
\]
as follows.  
\begin{itemize}
\item Define $\phi_I \colon H_1(F,S_+) \to \Z$ to be the composition
\[
\phi_I := H_1(F,S_+) \xrightarrow{\partial} H_0(S_+) \xrightarrow{[I]^* \cdot -} \Z
\]
where $[I]^*$ is the class of $H^0(S_+)$ dual to $[I] \in H_0(S_+)$.

\item For $k \geq 1$, define $\Phi_I \colon T^k H_1(F,S_+) \to T^{k-1} H_1(F,S_+)$ by
\begin{equation}\label{eq:PhiIDef}
\begin{aligned}
\Phi_I &:= \sum_{i=1}^k \bigg( T^{i-1}H_1(F,S_+) \otimes_{\Z} H_1(F,S_+) \otimes_{\Z} T^{k-i} H_1(F,S_+) \\
& \xrightarrow{(-1)^{i-1} \id_{T^{i-1}H_1(F,S_+)} \otimes \phi_I \otimes \id_{T^{i-1}H_1(F,S_+)}} T^{i-1}H_1(F,S_+) \otimes_{\Z} \Z \otimes_{\Z} T^{k-i} H_1(F,S_+). \bigg)
\end{aligned}
\end{equation}

\item Because of the sign $(-1)^{i-1}$ in \eqref{eq:PhiIDef}, we get an induced map
\[
\overline{\Phi_I} \colon \wedge^k H_1(F,S_+) \to \wedge^{k-1} H_1(F,S_+). 
\]
Define
\[
\overline{\Phi_I}' \colon \Zb^{S_+}_{\delta,\pi}(F) \to \Zb^{S_+}_{\delta,\pi}(F)
\]
by
\[
\overline{\Phi_I}'(\varepsilon_F \otimes \omega) := (-1)^{\pi(F)} \varepsilon_F \otimes \overline{\Phi_I}(\omega).
\]

\item Let $E$ be the sum of the maps $\overline{\Phi_I}'$ over all $k \geq 1$. $E$ is an odd map of degree $-1$ and the sign $(-1)^{i-1}$ in the definition of $\Phi_I$ ensures that $E^2 = 0$.
\end{itemize}
Define $E$ similarly when $I$ is an incoming interval component of $S_+$, except that:
\begin{itemize}
\item Instead of $(-1)^{i-1}$ we have $(-1)^{k-i}$ as the sign in \eqref{eq:PhiIDef};
\item We define
\[
\overline{\Phi_I}' \colon \Zb^{S_+}_{\delta,\pi}(F) \to \Zb^{S_+}_{\delta,\pi}(F)
\]
by
\[
\overline{\Phi_I}'(\varepsilon_F \otimes \omega) := \varepsilon_F \otimes \overline{\Phi_I}(\omega),
\]
without a sign of $(-1)^{\pi(F)}$.
\end{itemize}
\end{definition}

\begin{remark}
Because of how we chose the signs above, the actions of $\Z[E]/(E^2)$ on $\Zb^{S_+}_{\delta,\pi}(F)$ for outgoing interval components of $S_+$ are left actions and the actions for incoming interval components of $S_+$ are right actions. Correspondingly, we will think of $F$ as a cobordism ``from right to left'' in the open-closed cobordism category.
\end{remark}

\begin{example}\label{ex:ComputingE}
For simplicity, choose $(N_1,N_2,N_3,N_4)$ such that for the sutured surface $F$ shown in Figure~\ref{fig:EHomologyActionOriented} (adapted from \cite[Figure 12]{ManionDHA}), the parity shift $\pi(F)$ is even and correspondingly $\varepsilon_F$ is even. If we let $\alpha_i$ for $1 \leq i \leq 4$ denote the arcs and circles shown in Figure~\ref{fig:EHomologyActionOriented}, then we have
\begin{align*}
E_I(\varepsilon_F \otimes (\alpha_1 \wedge \alpha_2 \wedge \alpha_3 \wedge \alpha_4)) &= \varepsilon_F \otimes (-\alpha_1 \wedge \alpha_2 \wedge \alpha_4 + \alpha_1 \wedge \alpha_3 \wedge \alpha_4 + \alpha_1 \wedge \alpha_2 \wedge \alpha_4) \\
&= \varepsilon_F \otimes (\alpha_1 \wedge \alpha_3 \wedge \alpha_4).
\end{align*}
We could have done the same computation more directly by writing $\alpha_3$ as a circle rather than an arc with both endpoints on $I$ (in $H_1(F,S_+)$ there is no difference between these), so that the ``remove $\alpha_3$'' and ``remove $-\alpha_3$'' terms of $E_I$ do not appear at all. We drew $\alpha_3$ as shown in Figure~\ref{fig:EHomologyActionOriented} because such arcs often appear in bordered Heegaard Floer homology.
\end{example}

\begin{figure}
    \centering
    \includegraphics{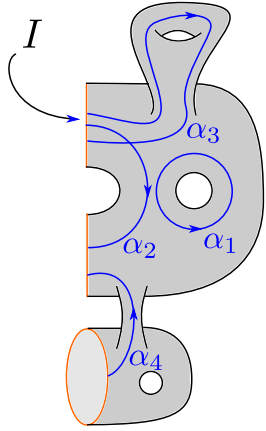}
    \caption{Examples of oriented arcs and circles in a sutured surface; from this data one can form the element $\varepsilon_F \otimes (\alpha_1 \wedge \alpha_2 \wedge \alpha_3 \wedge \alpha_4)$ of $\Zb^{S_+}_{\delta,\pi}(F)$, and Example~\ref{ex:ComputingE} computes $E_I$ applied to this element. Visually one can think of this application as summing, over all arc endpoints incident with $I$, of the removal (with appropriate signs, namely $+$ for an inward-pointing arc and $-$ for an outward-pointing arc after doing commutations that move $E_I$ inward from the left) of the corresponding arc from the wedge product.}
    \label{fig:EHomologyActionOriented}
\end{figure}

The same sign phenomenon responsible for $E^2 = 0$ ensures that if $I$ and $I'$ are two distinct incoming intervals or two distinct outgoing intervals of $S_+$, then $E_I$ and $E_{I'}$ anticommute. If $I$ is incoming and $I'$ is outgoing or vice-versa, then $E_I$ and $E_{I'}$ commute. 

If we let $m_{\inn}$ denote the number of incoming interval components and $m_{\out}$ denote the number of outgoing interval components of $S_+$, we have $A(M_1) = (\Z[E]/(E^2))^{\otimes m_{\inn}}$ and $A(M_2) = (\Z[E]/(E^2))^{\otimes m_{\out}}$. We thus get the structure of a (left, right) bimodule over the superalgebras $\left(A(M_1), A(M_2)\right)$ on $\Zb^{S_+}_{\delta,\pi}(F)$, where the superalgebra structure on $(\Z[E]/(E^2))^{\otimes m}$ is the tensor product of superalgebras (so that odd elements in different factors anticommute). 

\subsection{Tensor products}\label{sec:Tensor}

Here we discuss Example~\ref{ex:TensorShifts}, about the open $p$-tuple of pants $\Pc$, in more detail. Label the arcs in Figure~\ref{fig:OpenPantsBasis} from top to bottom as $e_1,\ldots,e_p$. Orient them as shown starting with $e_p$ which points from the outgoing boundary to the incoming boundary; $e_{p-1}$ should point from the incoming boundary to the outgoing boundary and so forth, in alternating fashion so that $e_1$ points from outgoing to incoming if $p$ is odd and from incoming to outgoing if $p$ is even. See Figure~\ref{fig:OpenPantsBasis}. As in Example~\ref{ex:TensorShifts}, make any choice of $(A_1,A_2,A_3,A_4)$ and $(N_1,N_2,N_3,N_4)$ with $A_1 = 1$ and $N_3 = 0$. Let $\delta = \delta_{A_1,A_2,A_3,A_4}$ and $\pi = \pi_{N_1,N_2,N_3,N_4}$.

\begin{figure}
    \centering
    \includegraphics[scale=0.8]{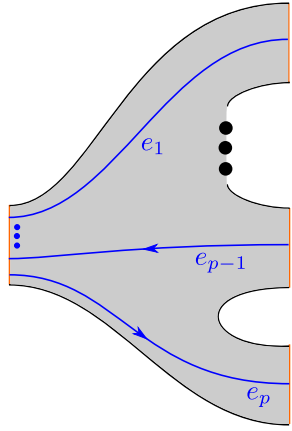}
    \caption{Orientations on arcs $e_1,\ldots,e_p$ giving a basis for $H_1(\Pc,S_+)$.}
    \label{fig:OpenPantsBasis}
\end{figure}

\begin{proposition}\label{prop:OpenPantsTensor}
Say we are given integers
\[
1 \leq i_1 < \cdots < i_k \leq p,
\]
and we label the elements of $\{1,\ldots,p\} \setminus \{i_1,\ldots,i_k\}$ as
\[
1 \leq j_1 < \cdots < j_{p-k} \leq p.
\]
For $1 \leq q \leq p-k$ let
\[
E_{j_q} := 1 \otimes \cdots 1 \otimes E \otimes 1 \otimes \cdots \otimes 1,
\]
an element of $(\Z[E]/(E^2))^{\otimes p}$; we can take products of such elements in $(\Z[E]/(E^2))^{\otimes p}$. The identification
\begin{equation}\label{eq:OpenPantsIso}
\varepsilon_{\Pc} \otimes (e_{i_1} \wedge \cdots \wedge e_{i_k}) \quad \leftrightarrow \quad E_{j_1} \cdots E_{j_{p-k}}
\end{equation}
gives an isomorphism 
\[
\Zc^{S_+}_{\delta,\pi}(\Pc) \cong (\Z[E]/(E^2))^{\otimes p}
\]
of $\Q$-graded (or $\Z$-graded) bimodules over $(\Z[E]/(E^2), (\Z[E]/(E^2))^{\otimes p})$, where the right action of $(\Z[E]/(E^2))^{\otimes p}$ on $(\Z[E]/(E^2))^{\otimes p}$ is by multiplication and the left action of $\Z[E]/(E^2)$ on $(\Z[E]/(E^2))^{\otimes p}$ is induced by the coproduct 
\[
\Delta(E) = E \otimes 1 + 1 \otimes E.
\]
\end{proposition}

\begin{proof}
    Example~\ref{ex:TensorShifts} ensures that the correspondence \eqref{eq:OpenPantsIso} respects degrees and parities; because it gives a bijection on basis elements, it is an isomorphism of $\Z$-graded super vector spaces. We want to check that \eqref{eq:OpenPantsIso} respects the left actions of $\Z[E]/(E^2)$ and the right actions of $(\Z[E]/(E^2))^{\otimes p}$.

    First we consider the left action of $E \in \Z[E]/(E^2)$. In $\Zb^{S_+}_{\delta,\pi}(\Pc)$,
    \begin{align*}
    E \cdot (\varepsilon_{\Pc} \otimes (e_{i_1} \wedge \cdots \wedge e_{i_k})) &= (-1)^{\pi({\Pc})} \varepsilon_{\Pc} \otimes (E \cdot (e_{i_1} \wedge \cdots \wedge e_{i_k})) \\
    &= (-1)^p \varepsilon_{\Pc} \otimes (E \cdot (e_{i_1} \wedge \cdots \wedge e_{i_k})).
    \end{align*}
    The action $E \cdot (e_{i_1} \wedge \cdots \wedge e_{i_k})$ is the sum, for $1 \leq r \leq k$, of
    \[
    (-1)^{r-1} e_{i_1} \wedge \cdots \wedge e_{i_{r-1}} \wedge E(e_{i_r}) \wedge e_{i_{r+1}} \wedge \cdots \wedge e_{i_k},
    \]
    and by our choice of orientation on $e_{i_r}$ we have 
    \[
    E(e_{i_r}) = (-1)^{p - i_r + 1} \cdot 1.
    \]
    It follows that $E \cdot (\varepsilon_{\Pc} \otimes (e_{i_1} \wedge \cdots \wedge e_{i_k}))$ is the sum, for $1 \leq r \leq k$, of
    \[
    (-1)^{i_r + r} e_{i_1} \wedge \cdots \wedge \widehat{e_{i_r}} \wedge \cdots \wedge e_{i_k},
    \]
    where $\widehat{e_{i_r}}$ means we remove the factor $e_{i_r}$ from the wedge product.

    Meanwhile, write
    \[
    E_{j_1} \cdots E_{j_{p-k}} = E \otimes \cdots \otimes E \otimes 1 \otimes E \otimes \cdots \cdots \cdots \otimes E \otimes 1 \otimes E \otimes \cdots \otimes E.
    \]
    When multiplying this expression by $E$ on the left, the first term will replace the $1$ in position $i_1$ with an $E$, picking up a sign of $i_1 - 1$. The second term will replace the $1$ in position $i_2$ with an $E$, picking up a sign of $i_2 - 2$. In general there will be $k$ terms, and for $1 \leq r \leq k$, the $r^{th}$ term will pick up a sign of $i_r - r$. Since
    \[
    (-1)^{i_r + r} = (-1)^{i_r - r},
    \]
    the identification \eqref{eq:OpenPantsIso} is compatible with the left actions of $\Z[E]/(E^2)$.

    Now consider the right action of $E$ in the $m^{th}$ factor of $(\Z[E]/(E^2))^{\otimes p}$, for $1 \leq m \leq p$; call this element $E_m$. We have
    \[
    (\varepsilon_{\Pc} \otimes (e_{i_1} \wedge \cdots \wedge e_{i_k})) \cdot E_m = 0
    \]
    unless $m = i_r$ for some $r$, and we have
    \begin{align*}
    &(\varepsilon_{\Pc} \otimes (e_{i_1} \wedge \cdots \wedge e_{i_k})) \cdot E_{i_r} \\
    &= (-1)^{k-r} \varepsilon_{\Pc} \otimes (e_{i_1} \wedge \cdots \wedge e_{i_{r-1}} \wedge (e_{i_r} \cdot E) \wedge e_{i_{r+1}} \wedge \cdots \wedge e_{i_k}).
    \end{align*}
    By our choice of orientation on $e_{i_r}$, we have $e_{i_r} \cdot E = (-1)^{p-i_r} \cdot 1$, so $(\varepsilon_F \otimes (e_{i_1} \wedge \cdots \wedge e_{i_k})) \cdot E_{i_r}$ equals
    \[
    (-1)^{p+k-i_r-r} \varepsilon_{\Pc} \otimes ( e_{i_1} \wedge \cdots \wedge \widehat{e_{i_r}} \wedge \cdots \wedge e_{i_k}).
    \]
    
    Meanwhile, writing $E_{j_1} \cdots E_{j_{n-k}}$ as a tensor product of $1$s and $E$s as above, we see that $(E_{j_1} \cdots E_{j_{n-k}}) \cdot E_m$ is also zero unless $m = i_r$ for some $r$. 
    \begin{itemize}
        \item For $r = k$, when computing $(E_{j_1} \cdots E_{j_{n-k}}) \cdot E_{i_k}$, to reach the $1$ that will be replaced by $E$, the odd element $E_{i_k}$ will need to commute past $p-i_k$ instances of $E$, picking up a sign of $(-1)^{p-i_k}$.
        \item For $r = k-1$, when computing $(E_{j_1} \cdots E_{j_{n-k}}) \cdot E_{i_{k-1}}$, the odd element $E_{i_{k-1}}$ will need to commute past $p - i_{k-1} - 1$ instances of $E$, picking up a sign of $(-1)^{k-i_{k-1} - 1}$. 
    \end{itemize}
    
    In general, when computing $(E_{j_1} \cdots E_{j_{n-k}}) \cdot E_{i_{k-q}}$, the odd element $E_{i_{k-q}}$ will need to commute past $p - i_{k-q} - q$ instances of $E$, picking up a sign of $(-1)^{p-i_{k-q} - q}$. Reindexing $k-q$ as $r$, we see that the sign in front of $(E_{j_1} \cdots E_{j_{n-k}}) \cdot E_{i_{r}}$ is $(-1)^{p-i_r - (k-r)}$, which equals $(-1)^{p+k-i_r-r}$ as desired.
\end{proof}

As mentioned in Example~\ref{ex:TensorShifts}, the correspondence \eqref{eq:OpenPantsIso} identifies the highest-degree element of $\Zb^{S_+}_{\delta,\pi}(\Pc)$ (defined up to a sign) with $\pm 1 \in (\Z[E]/(E^2))^{\otimes p}$.

\section{A gluing theorem with signs and gradings}

\subsection{The main gluing lemma}\label{sec:MainLemma}

The following is a version with signs and gradings of the main gluing lemma from \cite{ManionDHA}. Let $(F,\Lambda,S_+,S_-,\ell)$ be a sutured surface with all $S_+$ boundary components considered to be outgoing, and suppose that $I_1, I_2$ are interval components of $S_+$. As in \cite{ManionDHA}, there is a unique way up to homeomorphism to glue $I_1$ to $I_2$ and get an oriented surface $\overline{F}$, which can itself be viewed as a sutured surface with all $S_+$ boundary components outgoing (we will let $S_+$ denote the $S_+$ boundary of either $F$ or $\overline{F}$, determined by context). Let $\delta = \delta_{A_1,A_2,A_3,A_4}$ and $\pi = \pi_{N_1,N_2,N_3,N_4}$ for any $(A_1,A_2,A_3,A_4) \in \Q^4$ and $(N_1,N_2,N_3,N_4) \in \{0,1\}^4$. Let $E_1$ and $E_2$ denote the $E$-endomorphisms of $\Zb^{S_+}_{\delta,\pi}(F)$ arising from $I_1$ and $I_2$ respectively. 

\begin{lemma}[cf. Lemma 4.1 of \cite{ManionDHA}]\label{lem:MainGluing}
We have an isomorphism
\[
 \Zb^{S_+}_{\delta,\pi}(\overline{F})  \cong \frac{\Zb^{S_+}_{\delta,\pi}(F)}{\im(E_1 + E_2)}
\]
of $\Q$-graded super abelian groups, compatible with the left actions of $\Z[E]/(E^2)$ coming from $S_+$ intervals of $F$ that are not in $\{I_1,I_2\}$.
\end{lemma}

Before proving Lemma~\ref{lem:MainGluing} we will discuss a way of choosing bases for $\wedge^* H_1(F,S_+)$ following \cite[proof of Lemma 4.1]{ManionDHA}.

\begin{definition}\label{def:BasisChoice}
For a sutured surface $F$, choose a homeomorphism between $F$ and a finite disjoint union of ``standard'' sutured surfaces (Figure~\ref{fig:StandardSuturedSurface}), each consisting of a sphere connect-summed with some number of tori and with some number of open disks removed, with an even number of sutures on each boundary component. Then define a collection of arcs and circles in each standard sutured surface as follows.
\begin{itemize}
\item In each torus that was connect-summed on, take two circles giving a basis for the first homology.
\item For all of the boundary components intersecting $S_-$ nontrivially, except for one chosen component, take a circle around the boundary component.
\item Take a connected acyclic graph $\Gamma_F$ embedded in $F$ with vertices in $S_+$ (one vertex in each component of $S_+$).
\item Since we are working over $\Z$, we also need to choose orientations on the arcs and circles used to define this basis.
\end{itemize}
See Figure~\ref{fig:HomologyBasis} for an illustration. The circles and the edges of $\Gamma_F$ give a basis for $H_1(F,S_+)$, so subsets (corresponding to wedge products) of the set of arcs and circles give a basis for $\wedge^* H_1(F,S_+)$. The same construction gives a basis for $\Zb^{S_+}_{\delta,\pi}(F)$. We will refer to the circles and the edges of $\Gamma_F$ as basis arcs and basis circles.
\end{definition}

\begin{remark}
If we always have in mind the bases of Definition~\ref{def:BasisChoice}, which will be the case below, then the quantity $h = \rk H_1(F,S_+)$ is the number of basis arcs and circles. When self-gluing $F$ to get a new surface $\overline{F}$ below, to see how $h$ changes one can simply compare the number of basis arcs and circles before and after the self-gluing. Combinatorially, we have
\begin{equation}\label{eq:hFormula}
    \begin{aligned}
    h = &-2\#\{\textrm{components of }F\} + 2\#\{\textrm{genus}\} \\
    &+ \#\{\textrm{no-}S_+\textrm{ non-closed components of }F\} \\
    & + \#\{\textrm{no-}S_-\textrm{ non-closed components of }F\} \\
    &+ 2\#\{\textrm{closed components of }F\}\\
    & + \#\{S_+\textrm{ intervals}\} + \#\{S_+\textrm{ circles}\} \\
    & + \#\{S_-\textrm{ circles}\} +\#\{\textrm{boundary circles of }F\textrm{ with both }S_+\textrm{ and }S_-\}\\   
    \end{aligned}
\end{equation}
where by the genus of $F$ we mean the sum of the genus of each component of $F$.
\end{remark}

\begin{figure}
    \centering
    \includegraphics[scale=0.7]{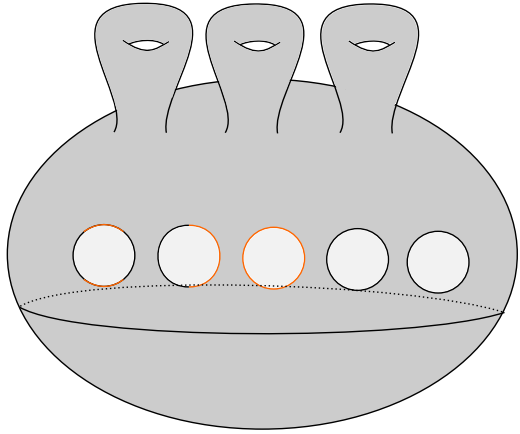}
    \caption{A standard (connected) sutured surface; on the boundary, $S_+$ is drawn in orange and $S_-$ is drawn in black. The sutures are the interface points between orange and black regions of the boundary.}
    \label{fig:StandardSuturedSurface}
\end{figure}

\begin{figure}
    \centering
    \includegraphics[scale=0.7]{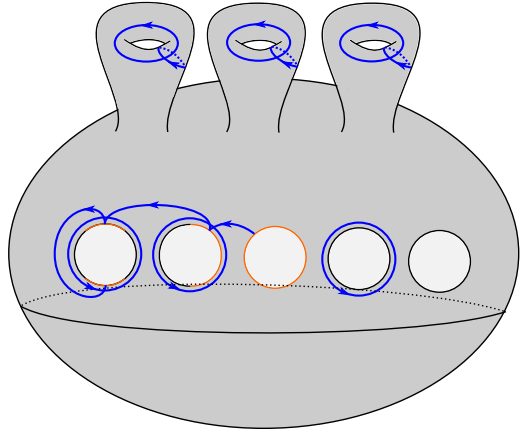}
    \caption{Examples of oriented arcs and circles used to specify a basis for $\wedge^* H_1(F,S_+)$. \textrm{   }}
    \label{fig:HomologyBasis}
\end{figure}

\begin{proof}[Proof of Lemma~\ref{lem:MainGluing}]
We will break the proof into cases.

\smallskip

\noindent \textbf{Case 1: $I_1$ and $I_2$ are on distinct components of $F$.} Assume that $I_1$ and $I_2$ live on distinct connected components of $F$. Begin by choosing a $\Z$-basis for $\wedge^* H_1(F,S_+)$ as in Definition~\ref{def:BasisChoice}; without loss of generality we may assume that $I_1$ and $I_2$ are each incident with at most one basis arc.

\smallskip

\noindent \textbf{Case 1-1: $I_1$ and $I_2$ are both alone.} Suppose that the components of $F$ containing $I_1$ and $I_2$ each contain no other components of $S_+$. In this case, no basis arcs are incident with either $I_1$ or $I_2$. We have $\delta(F) = \delta(\overline{F})$, $\pi(F) = \pi(\overline{F})$, and $H_1(F,S_+) \cong H_1(\overline{F},S_+)$. Both $E_1$ and $E_2$ act as zero on $\Zb^{S_+}_{\delta,\pi}(F)$, implying the statement of the lemma.

\smallskip

\noindent \textbf{Case 1-2: only $I_2$ is alone.} Now suppose exactly one of $\{I_1,I_2\}$ (without loss of generality $I_2$) is the unique component of $S_+$ in its component of $F$. In this case, $I_1$ is incident with one basis arc $e_1$, while $I_2$ is not incident with any basis arc. Without loss of generality, assume $e_1$ points from the surface into $I_1$.

The endomorphism $E_2$ of $\Zb^{S_+}_{\delta,\pi}(F)$ is zero. Let $\omega$ be a basis element for $\wedge^* H_1(F,S_+)$; the action of $E_1$ on $\varepsilon_F \otimes \omega$ is zero unless $e_1$ is the $i^{th}$ wedge factor of $\omega$ for some $i$, in which case write
\[
\omega = \gamma_1 \wedge \cdots \wedge \gamma_{i-1} \wedge e_1 \wedge \gamma_{i+1} \wedge \cdots \wedge \gamma_k
\]
where the $\gamma_j$ are basis elements for $H_1(F,S_+)$. If we set
\[
\omega' = \gamma_1 \wedge \cdots \wedge \gamma_{i-1} \wedge \gamma_{i+1} \wedge \cdots \wedge \gamma_k,
\]
then $E_1(\varepsilon_F \otimes \omega) = (-1)^{\pi(F) + i-1} \varepsilon_F \otimes \omega'$. We see that the quotient in the statement of the lemma amounts to setting $\varepsilon_F \otimes \omega'$ equal to zero whenever $\omega'$ is a basis element of $\wedge^* H_1(F,S_+)$ not divisible by $e_1$. Thus, we can take the elements
\[
\varepsilon_F \otimes \left( e_1 \wedge \omega' \right),
\]
where $\omega'$ is not divisible by $e_1$, as a $\Z$-basis for the quotient.

Now if we take all the basis arcs and circles, and remove $e_1$, the remaining arcs and circles can be used to define a basis of $H_1(\overline{F}, S_+)$. The corresponding $\Z$-basis for $\Zb^{S_+}_{\delta,\pi}(\overline{F})$ is given by the set of elements $\varepsilon_{\overline{F}} \otimes \omega'$ where $\omega'$ is as above. The identification
\begin{equation}\label{eq:Case1-2Identification}
\varepsilon_F \otimes \left( e_1 \wedge \omega' \right) \quad \leftrightarrow \quad \varepsilon_{\overline{F}} \otimes \omega'
\end{equation}
gives a map from the quotient to $\Zb^{S_+}_{\delta,\pi}(\overline{F}))$.

When passing from $F$ to $\overline{F}$, the changes in quantities relevant for the definitions of $\delta$ and $\pi$ are as follows.
\begin{itemize}
    \item $h$ is decreased by $1$.
    \item The number of $S_+$ intervals is decreased by $2$.
\end{itemize}
The change in $\delta$ is thus $A_1 - (A_1 - 1) = +1$, and the change in $\pi$ is $-1 -2N_3 = 1$ modulo $2$. Thus, \eqref{eq:Case1-2Identification} is an even map of degree zero; it is an isomorphism of super abelian groups because it gives a bijection on basis elements.

To see that this isomorphism intertwines the remaining actions, let $I \notin \{I_1,I_2\}$ be an interval component of $S_+$. First assume $e_1$ is not incident with $I$. In this case, the endomorphism $E_I$ of $\frac{\Zb^{S_+}_{\delta,\pi}(F)}{\im (E_1 + E_2)}$ corresponding to $I$ sends a basis element $\varepsilon_F \otimes (e_1 \wedge \omega')$ to $(-1)^{\pi(F) + 1} \varepsilon_F \otimes \left( e_1 \wedge E_I(\omega')\right)$. Meanwhile, the endomorphism $E_I$ of $\Zb^{S_+}_{\delta,\pi}(\overline{F})$ sends $\varepsilon_{\overline{F}} \otimes \omega'$ to $(-1)^{\pi(\overline{F})} \varepsilon_{\overline{F}} \otimes E_I(\omega')$. Our identification sends
\[
(-1)^{\pi(F) + 1} \varepsilon_F \otimes \left( e_1 \wedge E_I(\omega')\right) \quad \leftrightarrow \quad (-1)^{\pi(\overline{F})} \varepsilon_{\overline{F}} \otimes E_I(\omega')
\]
because $\pi(\overline{F}) = \pi(F) - 1$. If $e_1$ is incident with $I$, the analysis is unchanged because the extra term we would get in $\Zb^{S_+}_{\delta,\pi}(F)$ vanishes in the quotient 
\[
\frac{\Zb^{S_+}_{\delta,\pi}(F)}{\im (E_1 + E_2)} = \frac{\Zb^{S_+}_{\delta,\pi}(F)}{\im (E_1)}.
\]

\smallskip

\noindent \textbf{Case 1-3: neither $I_1$ nor $I_2$ is alone.} Now we consider the generic case; for $j \in \{1,2\}$, $I_j$ is incident with a unique arc $e_j$ used to define the basis of $H_1(F,S_+)$. Orient $e_1$  so that it points from the surface into $I_1$, and orient $e_2$ so that it points from $I_2$ into the surface. Choose the other orientations arbitrarily. 

As a left module over $(\Z[E]/(E^2))^{\otimes 2}$ (with $j^{th}$ factor corresponding to $E_j$ for $j \in \{1,2\}$), $\Zb^{S_+}_{\delta,\pi}(F)$ is free with a basis given by the set of elements
\[
\varepsilon_F \otimes (e_1 \wedge e_2 \wedge \omega'),
\]
where $\omega'$ is a $\Z$-basis element for $\wedge^* H_1(F,S_+)$ with neither $e_1$ nor $e_2$ as factors. Applying $E_1$ (which ``removes $e_1$'') to $\varepsilon_F \otimes (e_1 \wedge e_2 \wedge \omega')$ gives 
\[
(-1)^{\pi(F)} \varepsilon_F \otimes (e_2 \wedge \omega'),
\]
applying $E_1$ to $\varepsilon_F \otimes (e_1 \wedge \omega')$ gives
\[
(-1)^{\pi(F)} \varepsilon_F \otimes \omega',
\]
and applying $E_1$ to $\varepsilon_F \otimes (e_2 \wedge \omega')$ or $\varepsilon_F \otimes \omega'$ gives zero. Similarly, applying $E_2$ (which ``removes $-e_2$'') to $\varepsilon_F \otimes (e_1 \wedge e_2 \wedge \omega')$ gives
\[
(-1)^{\pi(F)} \varepsilon_F \otimes (e_1 \wedge \omega'),
\]
applying $E_2$ to $\varepsilon_F \otimes (e_2 \wedge \omega')$ gives
\[
(-1)^{\pi(F) + 1} \varepsilon_F \otimes \omega',
\]
and applying $E_2$ to $\varepsilon_F \otimes (e_1 \wedge \omega')$ or $\varepsilon_F \otimes \omega'$ gives zero.

It follows that a $\Z$-basis for $\frac{\Zb^{S_+}_{\delta,\pi}(F)}{\im(E_1 + E_2)}$ can be specified by taking the elements $\varepsilon_F \otimes (e_1 \wedge e_2 \wedge \omega')$ along with the elements
\[
\varepsilon_F \otimes (e_1 \wedge \omega') = -\varepsilon_F \otimes (e_2 \wedge \omega').
\]

Now, if we let $e$ be the oriented arc in $\overline{F}$ obtained by concatenating $e_1$ and $e_2$, then a $\Z$-basis for $\Zb^{S_+}_{\delta,\pi}(\overline{F})$ is given by the set of elements $\varepsilon_{\overline{F}} \otimes (e \wedge \omega')$ and $\varepsilon_{\overline{F}} \otimes \omega'$ where $\omega'$ runs over the same set as above. Make the identifications
\[
\varepsilon_F \otimes (e_1 \wedge e_2 \wedge \omega') \quad \leftrightarrow \quad  \varepsilon_{\overline{F}} \otimes (e \wedge \omega')
\]
and
\[
\varepsilon_F \otimes (e_1 \wedge \omega') \quad \leftrightarrow \quad  \varepsilon_{\overline{F}} \otimes \omega';
\]
this is an even map of degree zero because $\pi(\overline{F}) = \pi(F) + 1$ (the same argument applies as in the previous case). We get an isomorphism of $\Z$-graded super abelian groups as claimed in the statement of the lemma.

To see that this isomorphism intertwines the remaining actions, let $I \notin \{I_1, I_2\}$ be an interval component of $S_+$.

\smallskip

\noindent \textbf{Case 1-3a: $I$ is disjoint from $\{e_1,e_2\}$.} First assume that neither $e_1$ nor $e_2$ is incident with $I$. In this case, the endomorphism $E_I$ of $\frac{\Zb^{S_+}_{\delta,\pi}(F)}{\im (E_1 + E_2)}$ corresponding to $I$ sends a basis element $\varepsilon_F \otimes (e_1 \wedge e_2 \wedge \omega')$ to 
\[
(-1)^{\pi(F)} \varepsilon_F \otimes \left( e_1 \wedge e_2 \wedge E_I(\omega')\right).
\]
It sends a basis element $\varepsilon_F \otimes (e_1 \wedge \omega')$ to 
\[
(-1)^{\pi(F) + 1} \varepsilon_F \otimes (e_1 \wedge E_I(\omega')).
\]
Meanwhile, the endomorphism $E_I$ of $\Zb^{S_+}_{\delta,\pi}(\overline{F}))$ sends $\varepsilon_{\overline{F}} \otimes (e \wedge \omega')$ to 
\[
(-1)^{\pi(\overline{F}) + 1} \varepsilon_{\overline{F}} \otimes (e \wedge E_I(\omega')) = (-1)^{\pi(F)} \varepsilon_{\overline{F}} \otimes (e \wedge E_I(\omega'))
\]
and it sends $\varepsilon_{\overline{F}} \otimes \omega'$ to 
\[
(-1)^{\pi(\overline{F})} \varepsilon_{\overline{F}} \otimes E_I(\omega') = (-1)^{\pi(F) + 1} \varepsilon_{\overline{F}} \otimes E_I(\omega').
\]
Our identification sends
\[
(-1)^{\pi(F)} \varepsilon_F \otimes \left( e_1 \wedge e_2 \wedge E_I(\omega')\right) \quad \leftrightarrow \quad (-1)^{\pi(F)} \varepsilon_{\overline{F}} \otimes (e \wedge E_I(\omega'))
\]
and
\[
(-1)^{\pi(F) + 1} \varepsilon_F \otimes (e_1 \wedge E_I(\omega')) \leftrightarrow (-1)^{\pi(F) + 1} \varepsilon_{\overline{F}} \otimes E_I(\omega'),
\]
so it intertwines the actions of $E_I$ on the two sides.

\smallskip

\noindent \textbf{Case 1-3b: $I$ is incident with $e_1$.} Now suppose $I$ is incident with $e_1$; because $e_1$ points from the surface into $I_1$, it must point from $I$ into the surface. The same is true for $e$ in $\overline{F}$.

Some terms of the action on $E_I$ on each side (where neither $e_1$ nor $e$ is removed) are similar to the ones in Case 1-3a and thus correspond to each other under our identification. We also have a new term of $E_I(\varepsilon_F \otimes (e_1 \wedge e_2 \wedge \omega'))$, namely the ``remove $(-e_1)$'' term
\[
(-1)^{\pi(F) + 1} \varepsilon_F \otimes (e_2 \wedge \omega') = (-1)^{\pi(F)} \varepsilon_F \otimes (e_1 \wedge \omega'),
\]
and a new term of $E_I(\varepsilon_{\overline{F}} \otimes (e \wedge \omega'))$, namely the ``remove $(-e)$'' term
\[
(-1)^{\pi(\overline{F}) + 1} \varepsilon_F \otimes \omega' = (-1)^{\pi(F)} \varepsilon_F \otimes \omega'.
\]
These terms correspond to each other under our identification.

\smallskip

\noindent \textbf{Case 1-3c: $I$ is incident with $e_2$.}

In this case $e_2$ and $e$ point from the surface into $I$, so the new terms are the ``remove $e_2$'' term of $E_I(\varepsilon_F \otimes (e_1 \wedge e_2 \wedge \omega'))$, namely
\[
(-1)^{\pi(F) + 1} \varepsilon_F \otimes (e_1 \wedge \omega'),
\]
and the ``remove $e$'' term of $E_I(\varepsilon_{\overline{F}} \otimes (e \wedge \omega'))$, namely
\[
(-1)^{\pi(\overline{F})} \varepsilon_{\overline{F}} \otimes \omega' = (-1)^{\pi(F) + 1} \varepsilon_{\overline{F}} \otimes \omega'.
\]
These terms correspond to each other under our identification.

\smallskip

\noindent \textbf{Case 2: $I_1$ and $I_2$ are on the same component of $F$.} We consider some further subcases.

\smallskip

\noindent \textbf{Case 2-1: $I_1$ and $I_2$ are on the same component of $\partial F$.} Let $\gamma$ be the component of $\partial F$ containing $I_1$ and $I_2$. When applying Definition~\ref{def:BasisChoice} to choose a basis for $H_1(F,S_+)$, choose $\gamma$ as the unique not-fully-$S_+$ boundary component of its component of $F$ that does not get a basis circle around it, and such that $I_1$ and $I_2$ are incident with a single basis arc each, say $e_1$ and $e_2$ respectively (it is possible that $e_1 = e_2$). Orient $e_1$ so it points from the surface into $I_1$ and orient $e_2$ so it points from $I_2$ into the surface; this makes sense even if $e_1 = e_2$.

\smallskip

\begin{figure}
    \centering
    \includegraphics[scale=0.82]{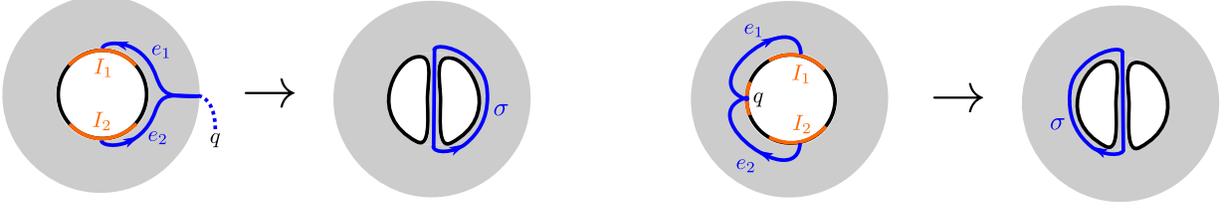}
    \caption{Local model near the circle $C$ for the arcs $e_1$ and $e_2$ and the circle $\sigma$ in Case 2-1a.}
    \label{fig:NearCircleSigns}
\end{figure}

\noindent \textbf{Case 2-1a: $I_1$ and $I_2$ are not alone in their component of $F$.} Assume the component of $F$ containing $I_1$ and $I_2$ also contains at least one other $S_+$ interval or $S_+$ circle. In this case we can assume that $e_1 \neq e_2$, and we can also assume that $e_1$ and $e_2$ share an endpoint $q$, as in \cite[Figure 19]{ManionDHA} or the analogue of that figure when $q$ is on the same component of $\partial F$ as $I_1$ and $I_2$. See Figure~\ref{fig:NearCircleSigns} here for illustrations of both possibilities.

With respect to $e_1$ and $e_2$, there are four types of $\Z$-basis elements for $\Zb^{S_+}_{\delta,\pi}(F)$, namely those of type $\varepsilon_F \otimes (e_1 \wedge e_2 \wedge \omega')$, $\varepsilon_F \otimes (e_1 \wedge \omega')$, $\varepsilon_F \otimes (e_2 \wedge \omega')$, and $\varepsilon_F \otimes \omega'$ where $\omega'$ is not divisible by $e_1$ or $e_2$. We have
\[
E_1(\varepsilon_F \otimes (e_1 \wedge e_2 \wedge \omega')) = (-1)^{\pi(F)} \varepsilon_F \otimes (e_2 \wedge \omega')
\]
because $E_1$ removes $e_1$ from the left, and
\[
E_2(\varepsilon_F \otimes (e_1 \wedge e_2 \wedge \omega')) = (-1)^{\pi(F)} \varepsilon_F \otimes (e_1 \wedge \omega')
\]
because $E_2$ removes $(-e_2)$ from the left. Thus, in the quotient $\frac{\Zb^{S_+}_{\delta,\pi}(F)}{\im(E_1 + E_2)}$, we have
\[
\varepsilon_F \otimes (e_2 \wedge \omega') = -\varepsilon_F \otimes (e_1 \wedge \omega')
\]
as well as
\[
\varepsilon_F \otimes \omega' = 0.
\]

Meanwhile, to define $\Gamma_{\overline{F}}$, remove $e_1$ and $e_2$ as basis arcs and add a basis circle $\sigma$ oriented compatibly with $e_1$ and $e_2$ (see Figure~\ref{fig:NearCircleSigns}).  There are two types of basis elements of $\Zb^{S_+}_{\delta,\pi}(\overline{F})$, namely those of type $\varepsilon_F \otimes (\sigma \wedge \omega')$ and $\varepsilon_F \otimes \omega'$ where $\omega'$ is not divisible by $\sigma$. Make the identifications
\begin{equation}\label{eq:Case2-1aIdentificationA}
\varepsilon_F \otimes (e_1 \wedge e_2 \wedge \omega') \quad \leftrightarrow \quad \varepsilon_{\overline{F}} \otimes (\sigma \wedge \omega')
\end{equation}
and
\begin{equation}\label{eq:Case2-1aIdentificationB}
\varepsilon_F \otimes (e_1 \wedge \omega') \quad \leftrightarrow \quad \varepsilon_{\overline{F}} \otimes \omega'.
\end{equation}

When passing from $F$ to $\overline{F}$, the changes in quantities relevant for the definitions of $\delta$ and $\pi$ are the same as the ones described above. The change in $\delta$ is thus $A_1 - (A_1 - 1) = +1$, and the change in $\pi$ is $-1 -2N_3 = 1$ modulo $2$. Because $\pi(\overline{F}) = \pi(F) + 1$, the identifications \eqref{eq:Case2-1aIdentificationA} and \eqref{eq:Case2-1aIdentificationB} give an even map of degree zero and thus an isomorphism of $\Z$-graded super abelian groups. The relation $\pi(\overline{F}) = \pi(F) + 1$ also ensures that this identification intertwines the actions of $E_{I'}$ for any $S_+$ interval component $I'$ not in $\{I_1, I_2\}$, except possibly the interval $I$ containing $q$ (if $q$ is contained in an $S_+$ interval rather than an $S_+$ circle). If $q$ is contained in an $S_+$ interval $I$, we have
\begin{align*}
&E_I(\varepsilon_F \otimes (e_1 \wedge e_2 \wedge \omega')) \\ 
&= (-1)^{\pi(F)} \varepsilon_F \otimes (e_2 \wedge \omega') + (-1)^{h\pi(F) + 2} \varepsilon_F \otimes (e_1 \wedge \omega') + (-1)^{\pi(F) + 2} \varepsilon_F \otimes (e_1 \wedge e_2 \wedge E_I(\omega')) \\
&= (-1)^{\pi(F)} \left(- \varepsilon_F \otimes (e_1 \wedge \omega') + \varepsilon_F \otimes (e_1 \wedge \omega') + \varepsilon_F \otimes (e_1 \wedge e_2 \wedge E_I(\omega')) \right) \\
&= (-1)^{\pi(F)} \varepsilon_F \otimes (e_1 \wedge e_2 \wedge E_I(\omega'))
\end{align*}
and
\[
E_I(\varepsilon_F \otimes (e_1 \wedge \omega')) = (-1)^{\pi(F) + 1} \varepsilon_F \otimes (e_1 \wedge E_I(\omega'))
\]
(note that $\varepsilon_F \otimes \omega'$ is zero in our quotient). On the $\overline{F}$ side, we have
\[
E_I(\varepsilon_{\overline{F}} \otimes (\sigma \wedge \omega')) = (-1)^{\pi(\overline{F}) + 1} \varepsilon_{\overline{F}} \otimes (\sigma \wedge E_I(\omega'))
\]
and
\[
E_I(\varepsilon_{\overline{F}} \otimes \omega') = (-1)^{\pi(\overline{F})} \varepsilon_{\overline{F}} \otimes E_I(\omega').
\]
Thus, our identification intertwines the actions of $E_I$.

\smallskip

\noindent \textbf{Case 2-1b: $I_1$ and $I_2$ are alone in their component of $F$.} Now assume that $I_1$ and $I_2$ are the only components of $S_+$ in their component of $F$; we have $e_1 = e_2 =: e$. Our chosen arcs and circles near $\gamma$ in $F$ and $\overline{F}$ are shown in Figure~\ref{fig:NearCircleSpecialSigns}.

\begin{figure}
    \centering
    \includegraphics[scale=0.84]{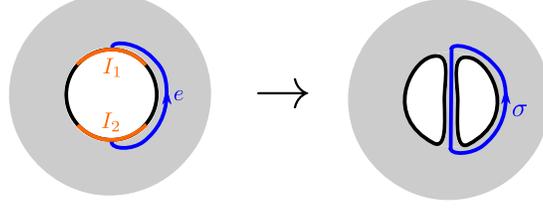}
    \caption{Local model near the circle $\gamma$ for the arc $e$ and the circle $\sigma$ in Case 2-1b.}
    \label{fig:NearCircleSpecialSigns}
\end{figure}

There are two types of basis elements for $\Zb^{S_+}_{\delta,\pi}(F)$, namely $\varepsilon_F \otimes (e \wedge \omega')$ and $\varepsilon_F \otimes \omega'$ where $\omega'$ is not divisible by $e$. We have
\[
E_1(\varepsilon_F \otimes (e \wedge \omega')) = (-1)^{\pi(F)} \varepsilon_F \otimes \omega'
\]
because $E_1$ removes $e$ from the left and
\[
E_2(\varepsilon_F \otimes (e \wedge \omega')) = (-1)^{\pi(F) + 1} \varepsilon_F \otimes \omega'
\]
because $E_2$ removes $(-e)$ from the left. These sum to zero, so the denominator in the quotient expression $\frac{\Zb^{S_+}_{\delta,\pi}(F)}{\im(E_1 + E_2)}$ is trivial. Meanwhile, there are two types of basis elements for $\Zb^{S_+}_{\delta,\pi}(\overline{F})$, namely $\varepsilon_{\overline{F}} \otimes (\sigma \wedge \omega')$ and $\varepsilon_{\overline{F}} \otimes \omega'$. Make the identifications
\begin{equation}\label{eq:Case2-1bIdentificationA}
\varepsilon_F \otimes (e \wedge \omega') \quad \leftrightarrow \quad \varepsilon_{\overline{F}} \otimes (\sigma \wedge \omega')
\end{equation}
and
\begin{equation}\label{eq:Case2-1bIdentificationB}
\varepsilon_F \otimes \omega' \quad \leftrightarrow \quad \varepsilon_{\overline{F}} \otimes \omega'.
\end{equation}

When passing from $F$ to $\overline{F}$, the changes in quantities relevant for the definitions of $\delta$ and $\pi$ are as follows.
\begin{itemize}
    \item The number of $S_+$ intervals is decreased by $2$.
    \item The number of no-$S_+$ non-closed components of $F$ is increased by $1$.
\end{itemize}
The change in $\delta$ is thus $(A_1 - 1) - (A_1 - 1) = 0$, and the change in $\pi$ is also zero. Thus, \eqref{eq:Case2-1bIdentificationA} and \eqref{eq:Case2-1bIdentificationB} give an even map of degree zero which is an isomorphism of super abelian groups because it gives a bijection on basis elements. The remaining actions of $E_I$ for $S_+$ intervals $I \notin \{I_1,I_2\}$ all take place in components of $F$ that are disjoint from the component containing $I_1$ and $I_2$, so our identification intertwines these remaining actions.

\smallskip

\noindent \textbf{Case 2-2: $I_1$ and $I_2$ are on different components of $\partial F$.} Let $\gamma_1$ and $\gamma_2$ be the components of $\partial F$ containing $I_1$ and $I_2$ respectively. In the standard model of Figure~\ref{fig:StandardSuturedSurface}, we may draw $\gamma_1$ and $\gamma_2$ as the two leftmost boundary circles.

\smallskip

\noindent \textbf{Case 2-2a: $I_1$ and $I_2$ are not alone in their component of $F$.} Assume the component of $F$ containing $I_1$ and $I_2$ also contains at least one other $S_+$ interval or $S_+$ circle; we can thus assume that $e_1 \neq e_2$. We can also assume that $e_1$ and $e_2$ share a vertex $q$ and that $\gamma_1$ is the chosen not-fully-$S_+$ boundary component in this component of $F$ where we do not draw a basis circle around $\gamma_1$. Let $\sigma$ be the basis circle around $\gamma_2$; see Figure~\ref{fig:AddingGenusSigns}.

\begin{figure}
    \centering
    \includegraphics[scale=0.78]{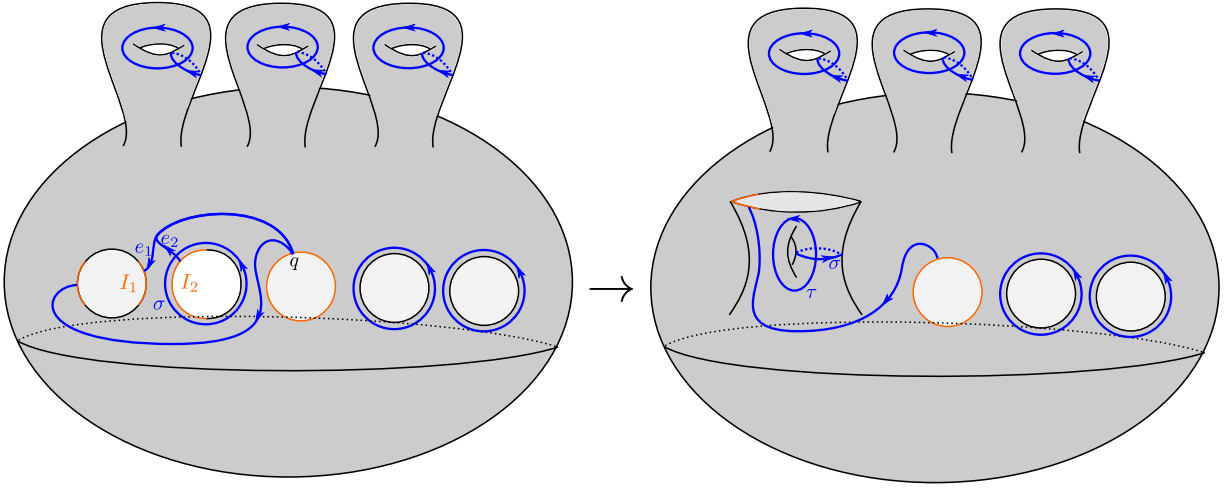}
    \caption{Before and after gluing $I_1$ to $I_2$, when $I_1$ and $I_2$ are on the same component of $F$ but different components of $\partial F$.}
    \label{fig:AddingGenusSigns}
\end{figure}

We partition the basis elements of $\Zb^{S_+}_{\delta,\pi}(F)$ into four types, namely $\varepsilon_F \otimes (e_1 \wedge e_2 \wedge \omega')$, $\varepsilon_F \otimes (e_1 \wedge \omega')$, $\varepsilon_F \otimes (e_2 \wedge \omega')$, and $\varepsilon_F \otimes \omega'$ where $\omega$ is not divisible by $e_1$ or $e_2$ (but may be divisible by $\sigma$). In the quotient $\frac{\Zb^{S_+}_{\delta,\pi}(F)}{\im(E_1 + E_2)}$, we thus have basis elements of type
\[
\varepsilon_F \otimes (e_1 \wedge e_2 \wedge \omega')
\]
and
\[
\varepsilon_F \otimes (e_1 \wedge \omega') = -\varepsilon_F \otimes (e_2 \wedge \omega').
\]
Meanwhile, we choose a basis for $\Zb^{S_+}_{\delta,\pi}(\overline{F})$ by picking a homeomorphism from $\overline{F}$ to the standard sutured surface shown in Figure~\ref{fig:AddingGenusSigns}, sending $\sigma$ in $F$ to the circle $\sigma$ in $\overline{F}$. While the graph $\Gamma_{\overline{F}}$ can be taken to be $\Gamma_F$ with $e_1$ and $e_2$ removed, we also need an additional closed circle $\tau$ to define the basis of $\Zb^{S_+}_{\delta,\pi}(\overline{F})$ because our gluing operation increased the genus by one. The orientations on $e_1$ and $e_2$ give us an orientation on $\tau$ as depicted in Figure~\ref{fig:AddingGenusSigns}. We have two types of basis elements $\varepsilon_{\overline{F}} \otimes (\tau \wedge \omega')$ and $\varepsilon_{\overline{F}} \otimes \omega'$ for  $\Zb^{S_+}_{\delta,\pi}(\overline{F})$, where $\omega'$ is not divisible by $\tau$ but may be divisible by $\sigma$. Make the identifications
\begin{equation}\label{eq:Case2-2IdentificationA}
\varepsilon_F \otimes (e_1 \wedge e_2 \wedge \omega') \quad \leftrightarrow \quad \varepsilon_{\overline{F}} \otimes (\tau \wedge \omega')
\end{equation}
and
\begin{equation}\label{eq:Case2-2IdentificationB}
\varepsilon_F \otimes (e_1 \wedge \omega') \quad \leftrightarrow \quad \varepsilon_{\overline{F}} \otimes \omega'.
\end{equation}

When passing from $F$ to $\overline{F}$, the changes in quantities relevant for the definitions of $\delta$ and $\pi$ are as follows.
\begin{itemize}
    \item $h$ is decreased by $1$.
    \item The number of $S_+$ intervals is decreased by $2$.
\end{itemize}
The change in $\delta$ is thus $A_1 - (A_1 - 1) = +1$, and the change in $\pi$ is $-1 -2N_3 = 1$ modulo $2$. Thus, \eqref{eq:Case2-2IdentificationA} and \eqref{eq:Case2-2IdentificationB} give an even map of degree zero, which is an isomorphism of super abelian groups because it gives a bijection on basis elements. The isomorphism intertwines all actions of $E_{I'}$  for any $S_+$ interval component $I'$ not in $\{I_1,I_2\}$, except possibly the interval $I$ containing $q$ (if $I$ is an interval and not a circle). If $q$ is in an $S_+$ interval $I$, we have
\begin{align*}
&E_I(\varepsilon_F \otimes (e_1 \wedge e_2 \wedge \omega')) \\
&= (-1)^{\pi(F)+1}(\varepsilon_F \otimes (e_2 \wedge \omega')) + (-1)^{\pi(F) + 1}(\varepsilon_F \otimes (e_1 \wedge \omega')) \\
&\quad+ (-1)^{\pi(F) + 2}\varepsilon_F \otimes (e_1 \wedge e_2 \wedge E_I(\omega')) \\
&= (-1)^{\pi(F)} \varepsilon_F \otimes (e_1 \wedge e_2 \wedge E_I(\omega'))
\end{align*}
and
\begin{align*}
&E_I(\varepsilon_F \otimes (e_1 \wedge \omega')) \\
&= (-1)^{\pi(F) + 1} \varepsilon_F \otimes \omega' + (-1)^{h_F + 1} \varepsilon_F \otimes (e_1 \wedge E_I(\omega')) \\
&= (-1)^{\pi(F)+ 1} \varepsilon_F \otimes (e_1 \wedge E_I(\omega')).
\end{align*}
Meanwhile, in $\overline{F}$ we have
\[
E_I(\varepsilon_{\overline{F}} \otimes (\tau \wedge \omega') = (-1)^{\pi(\overline{F}) + 1} \varepsilon_{\overline{F}} \otimes (\tau \wedge E_I(\omega'))
\]
and
\[
E_I(\varepsilon_{\overline{F}} \otimes \omega') = (-1)^{\pi(\overline{F})} \varepsilon_{\overline{F}} \otimes E_I(\omega').
\]
The relation $\pi(\overline{F}) = \pi(F) + 1$ ensures that our identification intertwines the actions of $E_I$.

\smallskip

\noindent \textbf{Case 2-2b: $I_1$ and $I_2$ are alone in their component of $F$.} If there are no other $S_+$ intervals or $S_+$ circles in the component of $F$ containing $I_1$ and $I_2$, then in Figure~\ref{fig:AddingGenusSigns} we should replace $e_1$ and $e_2$ by a single arc $e$ from $I_2$ to $I_1$. The two possible types of basis element for $\Zb^{S_+}_{\delta,\pi}(F)$ are $\varepsilon_F \otimes (e \wedge \omega')$ and $\varepsilon_F \otimes \omega'$. We have
\[
E_1(\varepsilon_F \otimes (e \wedge \omega')) =  (-1)^{\pi(F)} \varepsilon_F \otimes \omega'
\]
and
\[
E_2(\varepsilon_F \otimes (e \wedge \omega')) =  (-1)^{\pi(F)+1} \varepsilon_F \otimes \omega',
\]
so $E_1 + E_2$ is the zero map and the denominator of $\frac{\Zb^{S_+}_{\delta,\pi}(F)}{\im(E_1 + E_2)}$ is trivial. Meanwhile, the two possible types of basis element for $\Zb^{S_+}_{\delta,\pi}(\overline{F})$ are $\varepsilon_{\overline{F}} \otimes (\tau \wedge \omega')$ and $\varepsilon_{\overline{F}} \otimes \omega'$; make the identifications
\begin{equation}\label{eq:Case2-2bIdentificationA}
\varepsilon_F \otimes (e \wedge \omega') \quad \leftrightarrow \quad \varepsilon_{\overline{F}} \otimes (\tau \wedge \omega')
\end{equation}
and
\begin{equation}\label{eq:Case2-2bIdentificationB}
\varepsilon_F \otimes \omega' \quad \leftrightarrow \quad \varepsilon_{\overline{F}} \otimes \omega'.
\end{equation}
When passing from $F$ to $\overline{F}$, the changes in quantities relevant for the definitions of $\delta$ and $\pi$ are as follows.
\begin{itemize}
    \item The number of $S_+$ intervals is decreased by $2$.
    \item The number of no-$S_+$ non-closed components of $F$ is increased by $1$.
\end{itemize}
The change in $\delta$ is thus $(A_1 - 1) - (A_1 - 1) = 0$, and the change in $\pi$ is also zero. Thus, \eqref{eq:Case2-2bIdentificationA} and \eqref{eq:Case2-2bIdentificationB} give an even map of degree zero which is an isomorphism of super abelian groups because it gives a bijection on basis elements. Since all $S_+$ interval components $I \notin \{I_1,I_2\}$ of $F$ live on a different component of $F$ than $I_1$ and $I_2$, the isomorphism intertwines the remaining actions of $E_I$.
\end{proof}

\subsection{Proof of Theorem~\ref{thm:MainIntro}}\label{sec:MainProof}

\begin{proof}[Proof of Theorem~\ref{thm:MainIntro}]
The right side of the isomorphism in the statement of Theorem~\ref{thm:MainIntro} can be written as the quotient
\[
\frac{\Zb^{S_+}_{\delta,\pi}(F') \otimes_{\Z} \Zb^{S_+}_{\delta,\pi}(F)}{\spann_{\Z}\{(\varepsilon_{F'} \otimes x)a \otimes (\varepsilon_{F} \otimes y) - (\varepsilon_{F'} \otimes x) \otimes a(\varepsilon_{F} \otimes y)\}}
\]
where, in the denominator, $a$ is an arbitrary multiplicative generator 
\[
a = 1 \otimes \cdots \otimes 1 \otimes E \otimes 1 \otimes \cdots \otimes 1
\]
of $A(M_2)$, $x$ is an arbitrary basis element of $\wedge^* H_1(F', S_+)$, and $y$ is an arbitrary basis element of $\wedge^* H_1(F, S_+)$ (we will let $S_+$ denote the $S_+$ boundary of any sutured surface, determined by context).

The $\Q$-graded super abelian group\footnote{The notation ``$\mathrm{Left}_{A(M_1)}$'' means that one should view the right action by the (super)commutative superalgebra $A(M_1)$ as a left action instead, where elements of $A(M_1)$ pick up a sign as they pass from the left side to the right side in order to act by the original right action. The notation ``$\mathrm{Right}_{A(M_1)}$'' below is similar.}
\[
\mathrm{Left}_{A(M_1)} \left[ \Zb^{S_+}_{\delta,\pi}(F') \otimes_{\Z} \Zb^{S_+}_{\delta,\pi}(F) \right]
\]
with left actions of $A(M_3)$ and $A(M_1)$ is isomorphic as a $\Q$-graded super abelian group to 
\[
\Zb^{S_+}_{\delta,\pi}((F' \sqcup F)_{\mathrm{left}}),
\]
where $(F' \sqcup F)_{\mathrm{left}}$ is $F' \sqcup F$ with all $S_+$ boundary components viewed as outgoing, via the map $\Phi$ sending
 \[
 \varepsilon_{F'} \otimes x \otimes \varepsilon_{F} \otimes y \mapsto (-1)^{|x| \pi(F)} \varepsilon_{(F' \sqcup F)_{\mathrm{left}}} \otimes (x \wedge y).
 \]
Note that $\Phi$ is compatible with the left actions of $E$ for outgoing intervals of $F'$, and it relates the left actions of $E$ for incoming intervals of $F$ by a minus\footnote{Schematically, if we apply an odd ``remove $O$'' operator on the left of the sequence $XXOXXX$ where both $X$ and $O$ are odd, we move past two $X$'s and get two minus signs before we remove the $O$ ``from the left.'' On the other hand, if we first move this operator all the way to the right (picking up six minus signs) and then move it back to the left until it can cancel the $O$ ``from the right,'' we get three more minus signs for a total of nine. Nine minus signs give a different overall sign than two minus signs, and this pattern holds in general.} sign.
 
 The map $\Phi$ sends $(\varepsilon_{F'} \otimes x)a \otimes (\varepsilon_{F} \otimes y)$ to
\[
(-1)^{(|x|+1)\pi(F)} \varepsilon_{(F' \sqcup F)_{\mathrm{left}}} \otimes (xa \wedge y)
\]
where $xa$ is still computed in $\Zb^{S_+}_{\delta,\pi}(F')$. In terms of the left action $\bullet_1$ (coming from $F'$) of $A(M_2)$ on $\Zb^{S_+}_{\delta,\pi}((F' \sqcup F)_{\mathrm{left}})$, we can write this element as
\[
-(-1)^{(|x|+1)\pi(F) + \pi(F') + \pi(F) + |x|} a \bullet_1 \left(\varepsilon_{(F' \sqcup F)_{\mathrm{left}}} \otimes (x \wedge y)\right),
\]
which simplifies to
\[
-(-1)^{|x|\pi(F) + \pi(F') + |x|} a \bullet_1 \left(\varepsilon_{(F' \sqcup F)_{\mathrm{left}}} \otimes (x \wedge y)\right);
\]
the extra minus sign comes from the same phenomenon as in the footnote above. Similarly, $\Phi$ sends 
\[
(\varepsilon_{F'} \otimes x) \otimes a(\varepsilon_F \otimes y) = (-1)^{\pi(F)} (\varepsilon_{F'} \otimes x) \otimes (\varepsilon_F \otimes ay)
\]
to
\[
(-1)^{\pi(F) + |x|\pi(F)} \varepsilon_{(F' \sqcup F)_{\mathrm{left}}} \otimes (x \wedge ay)
\]
where $ay$ is still computed in $\Zb^{S_+}_{\delta,\pi}(F)$. In terms of the left action $\bullet_2$ (coming from $F$) of $A(M_2)$ on $\Zb^{S_+}_{\delta,\pi}((F' \sqcup F)_{\mathrm{left}})$, we can write this element as
\[
(-1)^{\pi(F) + |x|\pi(F) + \pi(F') + \pi(F) + |x|} a \bullet_2 \left(\varepsilon_{(F' \sqcup F)_{\mathrm{left}}} \otimes (x \wedge y)\right),
\]
which simplifies to 
\[
(-1)^{|x|\pi(F) + \pi(F') + |x|} a \bullet_2 \left(\varepsilon_{(F' \sqcup F)_{\mathrm{left}}} \otimes (x \wedge y)\right),
\]

Thus, as $\Q$-graded super abelian groups, we can identify the right side of the isomorphism in the statement of the corollary with
\begin{equation}\label{eq:LeftSide}
\mathrm{Right}_{A(M_1)} \left[ \frac{\Zb^{S_+}_{\delta,\pi}((F' \sqcup F)_{\mathrm{left}})}{\spann_{\Z}\{a \bullet_1 \left(\varepsilon_{(F' \sqcup F)_{\mathrm{left}}} \otimes z \right) + a \bullet_2 \left(\varepsilon_{(F' \sqcup F)_{\mathrm{left}}} \otimes z\right)\}} \right]
\end{equation}
where, in the denominator, $a$ is an arbitrary multiplicative generator of $A(M_2)$ as above and $z$ is an arbitrary basis element of $\wedge^* H_1((F' \sqcup F)_{\mathrm{left}}, S_+)$. This identification respects the left actions of $E$ for outgoing intervals of $F'$, and it relates the right actions of $E$ for incoming intervals of $F$ by a minus sign.

The denominator of \eqref{eq:LeftSide} can be viewed as the sum, over all components of $M_2$ (each corresponding to a pair of intervals $\{I_1,I_2\}$ in the target of $(F' \sqcup F)_{\mathrm{left}}$), of the subspaces
\[
\im(E_1 + E_2)
\]
of the numerator of \eqref{eq:LeftSide}, where $E_1$ and $E_2$ are the endomorphisms of the numerator corresponding to $I_1$ and $I_2$ respectively. Let $\overline{F}$ denote the result of gluing together one such pair of intervals $\{I_1,I_2\}$ in $(F' \sqcup F)_{\mathrm{left}}$, both arising from a component $m$ of $M_2$. By Lemma~\ref{lem:MainGluing}, the quotient in \eqref{eq:LeftSide} is isomorphic to the quotient of $\Zb^{S_+}_{\delta,\pi}(\overline{F})$ by the sum, over all components $m'$ of $M_2$ other than $m$ (corresponding to pairs of intervals $I'_1$ and $I'_2$), of the subspaces $\im(E'_1 + E'_2)$ where $E'_1$ and $E'_2$ are the endomorphisms of $\Zb^{S_+}_{\delta,\pi}(\overline{F})$ corresponding to $I'_1$ and $I'_2$ respectively.

By induction on the number of components of $M_2$, we see that the quotient in \eqref{eq:LeftSide} is isomorphic to
\[
\Zb^{S_+}_{\delta,\pi}((F' \circ F)_{\mathrm{left}})
\]
compatibly with the left actions of $A(M_3)$ and $A(M_1)$. Thus, the right side of the isomorphism in the statement of the corollary is isomorphic to
\[
\mathrm{Right}_{A(M_1)} \left[\Zb^{S_+}_{\delta,\pi}((F' \circ F)_{\mathrm{left}})\right],
\]
compatibly with the left actions of $E$ for outgoing intervals of $F'$ and relating the right actions of $E$ for incoming intervals of $F$ by a minus sign, and thereby to
\[
\Zb^{S_+}_{\delta,\pi}(F' \circ F)
\]
compatibly with the left actions of $A(M_3)$ and the right actions of $A(M_1)$.
\end{proof}

\subsection{Proof of Corollary~\ref{cor:OpenTQFT}}\label{sec:CorollaryProof}

If $A$ and $B$ are superalgebras over $\Z$ and $f \colon A \to B$ is a homomorphism of superalgebras (an even map respecting multiplication, units, and any additional gradings), let $X_f$ be the bimodule over $(B,A)$ defined by $X_f := B$ with actions specified by
\[
b \cdot x \cdot a := bxf(a)
\]
for $b,x \in B$ and $a \in A$. The correspondence $f \mapsto X_f$ gives a functor from the category of superalgebras and homomorphisms to the category of superalgebras and bimodules, in which morphisms from $A$ to $B$ are isomorphism classes of bimodules over $(B,A)$.

\begin{definition}\label{def:SAlgZ}
We let $\SAlg_{\Z}$ denote the symmetric monoidal category with:
\begin{itemize}
\item Objects: $\Z$-graded superalgebras over $\Z$ (i.e. $\Z$-graded super rings).
\item Morphisms from $A$ to $B$: $\Q$-graded bimodules over $(B, A)$ up to isomorphism, with composition given by tensor product over the superalgebra in the middle and the identity on $A$ given by $A$ as a bimodule over itself.
\item Tensor product on objects: tensor product $A_1 \otimes_{\Z} A_2$ of superalgebras.
\item Tensor product on morphisms: if $X$ is a bimodule over $(B_1,A_1)$ and $Y$ is a bimodule over $(B_2,A_2)$, define their tensor product to be $X \otimes_{\Z} Y$ with bimodule structure specified by
\begin{align*}
&(b_1 \otimes b_2) \cdot (x \otimes y) \cdot (a_1 \otimes a_2) \\
&:= (-1)^{|b_2||x| +|a_1||y| + |b_2||a_1|} (b_1 \cdot x \cdot a_1) \otimes (b_2 \cdot y \otimes a_2).
\end{align*}
This tensor product operation is well-defined on isomorphism classes of bimodules.
\item Monoidal unit: $\Z$ as a superalgebra (purely even and concentrated in degree zero).
\item Associator $\alpha_{A_1,A_2,A_3}$: isomorphism class of the bimodule $X_{\alpha}$ where
\[
\alpha \colon A_1 \otimes (A_2 \otimes A_3) \to (A_1 \otimes A_2) \otimes A_3
\]
is the canonical isomorphism of superalgebras.
\item Left unitor $\lambda_A$: isomorphism class of the bimodule $X_{\lambda}$ where
\[
\lambda \colon \Z \otimes A \to A
\]
is the canonical isomorphism of superalgebras. Right unitors are defined similarly.
\item Symmetrizer $\sigma_{A_1,A_2}$: isomorphism class of the bimodule $X_{\sigma}$ where
\[
\sigma \colon A_1 \otimes A_2 \to A_2 \otimes A_1
\]
is the canonical isomorphism of superalgebras specified by 
\[
\sigma(a_a \otimes a_2) = (-1)^{|a_1||a_2|} a_2 \otimes a_1.
\]
\end{itemize}
The coherence conditions for $\SAlg_{\Z}$ can be checked as identities of homomorphisms between superalgebras; it follows that they hold when passing from homomorphisms $f$ to isomorphism classes of bimodules $X_f$.
\end{definition}

\begin{proof}[Proof of Corollary~\ref{cor:OpenTQFT}]
By Theorem~\ref{thm:MainIntro}, the assignments in the statement of the corollary are compatible with composition. By Example~\ref{ex:TensorShifts} and Proposition~\ref{prop:OpenPantsTensor} with $p=1$, the identity cobordism on a single interval is sent to the identity bimodule over $\Z[E]/(E^2)$; because $\delta$, $\pi$, and $H_1(F,S_+)$ are additive with respect to disjoint unions, it follows that the identity cobordism on an object $M$ is sent to the identity bimodule over $A(M)$ in general.

For the monoidal functor coherence maps, we need an isomorphism class of bimodules $\varepsilon$ over $(\Zb^{S_+}_{\delta,\pi}(\varnothing),\Z)$ where $\varnothing$ is the monoidal unit of $\Cob^{\ext}_{\open}$, and for all pairs of objects $(M,M')$ of $\Cob^{\ext}_{\open}$, we need an isomorphism class of bimodules $\mu_{M,M'}$ over $(\Zb^{S_+}_{\delta,\pi}(M \sqcup M'), \Zb^{S_+}_{\delta,\pi}(M) \otimes_{\Z} \Zb^{S_+}_{\delta,\pi}(M'))$. Since $\varnothing$ has no interval components, $\Zb^{S_+}_{\delta,\mu}(\varnothing)$ is the empty tensor product of copies of $\Z[E]/(E^2)$, which is the monoidal unit $\Z$ of $\SAlg_{\Z}$ and we can take $\varepsilon$ to be the isomorphism class of $\Z$ as a bimodule over itself. Let $M$ and $M'$ be objects of $\Cob^{\ext}_{\open}$; define $\mu_{M,M'}$ to be the isomorphism class of the bimodule $X_{\mu}$ where 
\[
\mu \colon A(M) \otimes A(M') \to A(M \sqcup M')
\]
is the forget-the-parentheses isomorphism. To show that $\mu_{M,M'}$ is natural in $M$ and $M'$, let $F$ and $F'$ be morphisms in $\Cob^{\ext}_{\open}$ from $M_1$ to $M_2$ and from $M'_1$ to $M'_2$ respectively. We want to show that
\[
X := (\mu_{M_2, M'_2}) \otimes_{A(M_2) \otimes A(M'_2)} (\Zb^{S_+}_{\delta,\pi}(F) \otimes_{\Z} \Zb^{S_+}_{\delta,\pi}(F')) 
\]
and
\[
Y := \Zb^{S_+}_{\delta,\pi}(F \sqcup F') \otimes_{A(M_1 \sqcup M'_1)} (\mu_{M_1, M'_1})
\]
are isomorphic as bimodules over $(A(M_2 \sqcup M'_2), A(M_1) \otimes A(M'_1))$. Indeed, the additivity of $\delta$, $\pi$, and $H_1(F,S_+)$ with respect to disjoint unions gives us an isomorphism of $\Q$-graded abelian groups; compatibility with the algebra actions follows from looking at each multiplicative generator $E_I$ of the algebras separately, so that parenthesization and the $\mu$ bimodules are irrelevant.

Compatibility between $\mu_{M,M'}$ and the associators of $\Cob^{\ext}_{\open}$ and $\SAlg_{\Z}$ holds at the level of the algebra homomorphisms underlying all maps in the compatibility square, where it amounts to saying that two forget-the-parentheses isomorphisms from $(A(M_1) \otimes A(M_2)) \otimes A(M_3)$ to $A(M_1 \sqcup (M_2 \sqcup M_3))$ are equal. Thus, it holds for the bimodules up to isomorphism. Compatibility between $\mu_{M,M'}$, $\varepsilon$, and the unitors of $\SAlg_{\Z}$ and $\Cob^{\ext}_{\open}$ also holds at the level of algebra homomorphisms; both ways around the left unitor compatibility square of homomorphisms give the left multiplication map from $\Z \otimes A(M)$ to $A(M)$, and similarly for the right unitor square of homomorphisms. Thus, compatibility with unitors holds at the level of bimodules up to isomorphism.

Finally, to see that $\mu_{M,M'}$ is compatible with the symmetrizers of $\SAlg_{\Z}$, let $M_1$ and $M_2$ be objects of $\Cob^{\ext}_{\open}$. The symmetrizer $\sigma_{M_1,M_2}$ of $M_1$ and $M_2$ in $\Cob^{\ext}_{\open}$ can be described as $(M_1 \sqcup M_2) \times [0,1]$ viewed as a cobordism from the $t=1$ slice $M_1 \sqcup M_2$ (in that order) on the right to the $t=0$ slice $M_2 \sqcup M_1$ (in that order) on the left. In other words, it is the identity cobordism on $M_1 \sqcup M_2$ but with the order of the disjoint union reversed in its target, which becomes $M_2 \sqcup M_1$. By Example~\ref{ex:TensorShifts} and Proposition~\ref{prop:OpenPantsTensor} with $p=1$, we can identify $\Zb^{S_+}_{\delta,\pi}(\sigma_{M_1,M_2})$ with $A(M_1) \otimes A(M_2)$ with its usual right action of $A(M_1) \otimes A(M_2)$ by multiplication and with left action of $a_2 \otimes a_1 \in A(M_2) \otimes A(M_1)$ given by multiplication by $(-1)^{|a_1||a_2|} a_1 \otimes a_2$. It follows that $\Zb^{S_+}_{\delta,\pi}(\sigma_{M_1,M_2})$ is the symmetrizer of $(A(M_1), A(M_2))$ in $\SAlg_{\Z}$ (more precisely, the compatibility square involving symmetrizers and the monoidal coherence isomorphisms commutes, but we have suppressed explicit mention of the monoidal coherence isomorphisms like $A(M_1 \sqcup M_2) \cong A(M_1) \otimes A(M_2)$).
\end{proof}

\section{Constraints on degree and parity}\label{sec:DegreeParityConstraints}

We conclude with a brief discussion of why we chose the particular parametrized expressions \eqref{eq:DeltaDef} and \eqref{eq:PiDef} for the degree and parity shifts. In Mikhaylov \cite[5.2.2]{Mikhaylov}, the degree shift on the state space built from $\wedge^* H^1$ is proportional to the dimension of $H^1$, and the same is true in \cite[Section 3.2]{GeerYoung}. Analogously, for us the quantity $h = \rk H_1(F,S_+)$ and its combinatorial formula give a reasonable starting point for defining degree and parity shifts. In \eqref{eq:hFormula}, $h$ depends on various pieces of data associated to a sutured surface:
\begin{itemize}
    \item Number $k_1$ of components
    \item Genus $k_2$ (sum over all components)
    \item Number $k_3$ of closed components
    \item Number $k_4$ of non-closed components without $S_+$
    \item Number $k_5$ of non-closed components without $S_-$
    \item Number $k_6$ of $S_+$ intervals
    \item Number $k_7$ of $S_+$ circles 
    \item Number $k_8$ of $S_-$ circles
    \item Number $k_9$ of boundary circles of $F$ with both $S_+$ and $S_-$
\end{itemize}
The above parameters $k_1,\ldots,k_9$ do not satisfy any linear relations as we vary the sutured surface, and on the other hand other relevant data can be computed from $k_1,\ldots,k_9$. For example, the number of components of $F$ intersecting both $S_+$ and $S_-$ is $k_1 - k_3 - k_4 - k_5$, the number of $S_-$ intervals is $k_6$, and the number of boundary components of $F$ is $k_7 + k_8 + k_9$. Similarly, \eqref{eq:hFormula} gives $h$ as a linear combination of the $k_i$.

From the proof of Lemma~\ref{lem:MainGluing}, one can deduce that if the grading shift function $\delta$ has a general formula sending $F$ to some linear combination
\[
C_1 k_1 + \cdots + C_9 k_9
\]
of the $k_i$, then $\delta$ is consistent with Lemma~\ref{lem:MainGluing} if and only if the following system of equations is satisfied:
\begin{align}
    -C_1 + C_4 - 2C_6 + C_8 - 2C_9 &= 0 \quad \textrm{ (Case 1-1)} \label{it:Req1} \\ 
    -C_1 -2C_6 -C_9 &= 1 \quad \textrm{ (Case 1-2 or 1-3, no $S_-$ circle created)} \label{it:Req2} \\
    -C_1 -2C_6 + C_8 - 2C_9 &= 1 \quad \textrm{ (Cases 1-2 or 1-3, one $S_-$ circle created)}\label{it:Req2.5} \\
    -2C_6 + C_9 &= 1 \quad \textrm{ (Case 2-1a, no $S_-$ circle created)} \label{it:Req3} \\
    -2C_6 + C_8 &= 1 \quad \textrm{ (Case 2-1a, one $S_-$ circle created)} \label{it:Req4} \\
    -2C_6 + 2C_8 - C_9 &= 1 \quad \textrm{ (Case 2-1a, two $S_-$ circles created)} \label{it:Req5} \\
    C_4 - 2C_6 + 2C_8 - C_9 &= 0 \quad \textrm{ (Case 2-1b)} \label{it:Req6} \\
    C_2 - 2C_6 - C_9 &= 1 \quad \textrm{ (Case 2-2a, no $S_-$ circle created)} \label{it:Req7} \\
    C_2 - 2C_6 + C_8 - 2C_9 &= 1 \quad \textrm{ (Case 2-2a, one $S_-$ circle created)} \label{it:Req8} \\
    C_2 + C_4 - 2C_6 + C_8 - 2C_9 &= 0 \quad \textrm{ (Case 2-2b)} \label{it:Req9}
\end{align}

Subtracting \eqref{it:Req3} and \eqref{it:Req4}, we get $C_8 = C_9$, so \eqref{it:Req2} and \eqref{it:Req2.5} become the same equation and \eqref{it:Req3}, \eqref{it:Req4}, and \eqref{it:Req5} become the same equation. This second equation can be solved to give $C_6 = (C_9 - 1)/2$. Subtracting \eqref{it:Req2} and \eqref{it:Req3}, we get $2C_9 + C_1 = 0$, so that $C_1 = - 2C_9$. Now \eqref{it:Req6} becomes
\[
C_4 - (C_9 - 1) + 2C_9 - C_9 = 0
\]
which gives $C_4 = -1$. Simplifying \eqref{it:Req1}, we get
\[
2C_9 - 1 - (C_9 - 1) + C_9 - 2C_9 = 0
\]
which is now tautological. For the final three equations, \eqref{it:Req7} and \eqref{it:Req8} both become
\[
C_2 - (C_9 - 1) - C_9 = 1
\]
or equivalently $C_2 = 2C_9$ ; substituting this result into \eqref{it:Req9} we get
\[
2C_9 - 1 - (C_9 - 1) + C_9 - 2C_9 = 0
\]
which is tautological. Thus, in terms of the coefficients $C_i$, to solve the above system of equations we can treat $C_3, C_5, C_7$, and $C_9$ as free parameters, with the other parameters determined by these four. Now, using \eqref{eq:hFormula} we have
\[
C_9 h = -2C_9 k_1 + 2C_9 k_2 +2C_9 k_3 + C_9 k_4 + C_9 k_5 + C_9 k_6 + C_9 k_7 + C_9 k_8 + C_9 k_9.
\]
We can thus express our linear combination
\[
-2C_9 k_1 + 2C_9 k_2 + C_3 k_3 - k_4 + C_5 k_5 + ((C_9 - 1)/2)k_6 + C_7 k_7 + C_9 k_8 + C_9 k_9
\]
as
\[
C_9 h + C_3 k_3 + (-C_9 - 1)k_4 + C_5 k_5 + ((-C_9-1)/2)k_6 + C_7 k_7
\]
where the parameters $C_3,C_5,C_7$ have absorbed the coefficients of $k_3,k_5,k_7$ in $h$. Rewriting $(C_3,C_5,C_7,C_9)$ as $(A_3, A_2, A_4, -A_1)$, our linear combination becomes
\begin{equation}\label{eq:DeltaDefAgain}
-A_1 h + A_3 k_3 + (A_1 - 1)k_4 + A_2 k_5 + ((A_1 - 1)/2)k_6 + A_4 k_7,
\end{equation}
which for $(A_1,A_2,A_3,A_4) \in \Q^4$ is the general degree shift compatible with Lemma~\ref{lem:MainGluing}. This computation motivated our definition of $\delta_{A_1,A_2,A_3,A_4}$ in \eqref{eq:DeltaDef}.

For the parity shift, one can ask when \eqref{eq:DeltaDefAgain} gives an element of $\Z \subset \Q$. If this is true for all sutured surfaces $F$, then taking $F$ to be a disk with boundary consisting of one $S_+$ interval and one $S_-$ interval, we get $(A_1 - 1)/2 \in \Z$, so $A_1 = 2N + 1$ for some $N \in \Z$. Substituting into \eqref{eq:DeltaDefAgain} gives
\[
-(2N + 1)h + A_3 k_3 + 2N k_4 + A_2 k_5 + N k_6 + A_4 k_7,
\]
which modulo $2$ is equal to
\[
-h + A_3 k_3 + A_2 k_5 + N k_6 + A_4 k_7.
\]
Taking $F = S^2$, we get $A_3 \in \Z$, taking $F = S^1 \times [0,1]$ with both boundary components in $S_+$ we get $A_2 \in \Z$, and taking $F$ to be $(S^1 \times [0,1]) \sqcup D^2$ with all boundary components in $S_+$ we get $A_4 \in \Z$. In this way we arrive at \eqref{eq:PiDef}.

\bibliographystyle{alpha}
\bibliography{biblio.bib}

\end{document}